\documentclass[a4]{article}

\usepackage{amsmath,amssymb,amsthm}
\usepackage[utf8]{inputenc}
\usepackage{xcolor}
\usepackage{pgf,tikz,color}
\usepackage{multirow}

\usepackage[colorlinks = true, citecolor = {blue}]{hyperref}

\usepackage{cleveref}

\title{An algebraic approach to Erd\H{o}s-Ko-Rado sets of flags in spherical buildings}
\author{Jan De Beule, Sam Mattheus, Klaus Metsch}
\date{}

\newtheorem{theorem}{Theorem}[section]
\newtheorem{prop}[theorem]{Proposition}
\newtheorem{cor}[theorem]{Corollary}
\newtheorem{lemma}[theorem]{Lemma}
\newtheorem{prob}[theorem]{Problem}

\theoremstyle{definition}
\newtheorem{defin}[theorem]{Definition}
\newtheorem{remark}[theorem]{Remark}
\newtheorem{exa}[theorem]{Example}
\newtheorem{algo}{Algorithm}
\newtheorem{property}[theorem]{Property}

\def\C{\mathbb C}
\def\PG{\mathrm{PG}}

\newcommand{\qbinom}[2]{\left[{{#1}\atop#2}\right]_q}

\def\GL{\mathrm{GL}}
\def\PGL{\mathrm{PGL}}

\def\PSp{\mathrm{PSp}}
\def\PGO+{\mathrm{PGO}^+}
\def\PGO-{\mathrm{PGO}^-}
\def\PGO{\mathrm{PGO}}

\def\PGU{\mathrm{PGU}}

\def\Sym{\mathrm{Sym}}

\begin{document}

\maketitle

\begin{abstract}
	In this paper, oppositeness in spherical buildings is used to define an EKR-problem for flags in projective and polar spaces. A novel application of the theory of buildings and Iwahori-Hecke algebras is developed to prove sharp upper bounds for EKR-sets of flags. In this framework, we can reprove and generalize previous upper bounds for EKR-problems in projective and polar spaces. The bounds are obtained by the application of the Delsarte-Hoffman coclique bound to the opposition graph. The computation of its eigenvalues is due to earlier work by Andries Brouwer and an explicit algorithm is worked out. For the classical geometries, the execution of this algorithm boils down to elementary combinatorics. Connections to building theory, Iwahori-Hecke algebras, classical groups and diagram geometries are briefly discussed. Several open problems are posed throughout and at the end.
\end{abstract}

\section{Introduction}

Ever since the original result on intersecting families of sets by Erd\H{o}s, Ko and Rado \cite{EKR61}, an abundance of similar `EKR-theorems' in different settings has appeared, with a fitting concept of `intersecting' in each. Along with these results came a plethora of techniques to prove them. Some of the most powerful ones, both in success rate and scope of applicability, lie in the field of algebraic combinatorics and algebraic graph theory in particular. Typically, the problem at hand is restated to the search of the largest cocliques in the graphs of the `non-intersection' relation, for which several tools exist. This approach is highlighted in the recent book by Godsil and Meagher \cite{GM16}, where EKR-theorems for sets, vector spaces, words, partitions and permutations are considered. Other instances where this technique has proven its worth are the results for polar spaces in \cite{IMM18,PSV11,Sta80}. \\

One property that these problems have in common is that they are `symmetrical', not in the sense of a large automorphism group, but because the natural relations between the sets or subspaces of the relevant geometries under consideration are symmetrical. This has implications for the relevant algebraic structures with which one intends to attack these problems. Commonly, the adjacency matrix of the corresponding graph generates a commutative matrix algebra, which is considerably easier to study as opposed to a non-commutative matrix algebra. 

When we drop the symmetry restriction, this nice feature is lost. And yet, researchers have investigated these kinds of problems as well, but with a different toolbox in hand. Notable examples in finite geometry are the papers due to Blokhuis and Brouwer \cite{BB17}, with Sz\H{onyi} \cite{BBS14}, with G\"uven \cite{BBG14} and the paper by the last author and Werner \cite{MW19}. All of these deal with flags in finite geometries, which are sets of pairwise incident elements of the geometry. When the flags have size one and consist of a single element, the natural relations are symmetrical, but as soon as we consider larger flags, we not only lose symmetry, but also the commutativity of the matrix algebras. However, the particular relation of `non-intersection', defined appropriately for each geometry, is still symmetrical. In this paper we will show how to overcome the non-commutativity of the matrix algebras when dealing with general flags in finite geometries and obtain upper bounds for EKR-sets of flags. \\

Some caution needs to be taken when comparing our results to previous papers. It is important to note each time what the relevant notion of `non-intersection' is, and how the corresponding EKR-problem is stated. In this paper, we will consider oppositeness of flags as defined in spherical buildings as our starting point, as explained in \Cref{section:opposition}. Opposition is typically stricter than other analogues of `non-intersection'. For example in \cite{Metsch19} `non-intersection' for subspaces of polar spaces is defined as the usual non-intersection of subspaces. However, non-intersecting subspaces in polar spaces are not necessarily opposite. Another common definition for flags in a projective space is `being in general position'. This is also different, as any two distinct lines in $\PG(2,q)$ are in general position, while they can never be opposite, as we will see later on. This is also made clear in the distinction between oppositeness graphs and Kneser graphs as in G\"uven's thesis \cite{Guven12}, where similar results are obtained. In fact, in the geometries we consider, we will see that opposition of flags can be defined as being in general position with an extra condition. The reason to consider this extra condition is the very nice algebraical consequence we discuss in \Cref{rem:partialflagslosefactor}. \\

The starting point of our investigation is \cite{Brouwer10}. The main theorem in this paper is that the eigenvalues of opposition in spherical buildings are powers of $q$, where $\mathbb{F}_q$ is the field over which the building is defined. The terminology of buildings is used as this is the natural framework to study the combinatorics of flags in finite geometries. Tits \cite{Tits74} showed that the same framework could be approached starting from the classical groups of Lie type (see \Cref{table:classicalgroups}) and considering their actions on flags in the corresponding geometries, at least when the rank of the geometry is at least three. We will take this point of view and avoid terminology from building theory where possible, in order to make the paper more self-contained.

Although it is not mentioned explicitly in \cite{Brouwer10}, an algorithm can be extracted by which one can compute all eigenvalues of opposition. This is the content of \Cref{algorithmmaximal} and \Cref{algorithmpartial}. We will feed these eigenvalues into the Delsarte-Hoffman bound for cocliques in the opposition graph and hence obtain an upper bound for EKR-sets of flags. In fact, to do so we only need the largest and smallest eigenvalue which, given the machinery, can be found by remarkably elementary combinatorics.

The main idea to compute these eigenvalues, which is also implicit in \cite{Brouwer10}, is that the oppositeness relation in a spherical building in fact corresponds to a generator of the Iwahori-Hecke algebra attached to the building. To be more precise: to every spherical building there is a finite Coxeter system $(W,S)$  associated. We can then define an algebra, called the Iwahori-Hecke algebra, with generators indexed by the elements of $W$ and depending on some parameters $\{q_s \,\, | \,\, s \in S\}$. This algebra can be seen as a deformation of the group algebra $\C W$ and has been studied over the last 60 years. A lot of its structure is known in the classical types, including its simple modules and hence its irreducible characters. The main point is that this Iwahori-Hecke algebra is in fact isomorphic to the non-commutative association scheme obtained from the group action of the classical groups on maximal flags. This means that we can obtain the eigenvalues of the opposition relation from the irreducible characters of the Iwahori-Hecke algebra. A recent and more detailed description of the connection between these two different points of view can be found in \cite[Section 4]{Guillot20}. 

Explaining all topics mentioned above in detail is the content of multiple books. We will hence not go in full detail, but try to provide a working knowledge in \Cref{section:generalities} which suffices to run \Cref{algorithmmaximal} and \Cref{algorithmpartial}. The interested reader can find more in-depth material about each of the topics in the following references: for buildings and finite groups with $BN$-pairs we refer to the Tits' book \cite{Tits74} or the more recent account by Garrett \cite{Garrett97}, the book of Geck and Pfeiffer \cite{GP00} for an excellent treatment of Iwahori-Hecke algebras and their representation theory and finally the treatise on diagram geometry \cite{BC13} by Buekenhout and Cohen. \\

In the following sections, we apply this theory to the classical spherical buildings of rank at least three, which are geometries related to the classical groups of Lie type. There are two reasons we restrict ourselves to these groups and not consider the exceptional groups. The first one is that the classical groups consist of infinite families corresponding to buildings of unbounded rank. Therefore, there is something to prove for general rank, as opposed to the exceptional groups. Secondly, while we will find sharp upper bounds for EKR-sets of flags in the classical spherical buildings, we doubt that the same method will produce sharp bounds for the buildings associated to exceptional groups. We will show an example at the end of the paper.

Furthermore, there are two more reasons for the restriction on the rank. As indicated before, buildings of rank three are known to be associated with groups of Lie type by the results of Tits, which means that it is equivalent to define the geometries coming from the groups or from an axiomatic building-theoretical point of view. Since we will try to keep the prerequisite knowledge of building theory to a minimum, we can content ourselves with the discussion of the geometries from the group-theoretical point of view. The second reason is that EKR-problems for buildings of rank two are easy exercises, and hence do not need to be dealt with here.

\section{Generalities}\label{section:generalities}

We first recall some general theory from association schemes in order to demonstrate how Tits' theory of groups with $BN$-pairs fits into the picture. Most of this information is also well explained in \cite[Chapter 10]{BCN89} albeit for flags of one element. We will not make this restriction and talk about the action of classical groups on flags of any type. Then we move on to the correspondence between the association scheme from the group action and the Iwahori-Hecke algebra associated to the classical group.

\subsection{Association schemes from group actions}\label{section:flagstocosets}

Consider a finite group $G$ acting transitively on a finite set $\Omega$. Then the orbitals, which are the orbits of the diagonal action of $G$ on $\Omega \times \Omega$ defined by $g(x,y) := (gx,gy)$, are the relations $\{R_1,\dots,R_m\}$ of an association scheme which we will denote by ${\cal{A}}(G,\Omega)$. One can check that these relations satisfy the axioms of an association scheme, which in general is not commutative. 



We can represent ${\cal{A}}(G,\Omega)$ in another way, entirely contained in $G$. Let $B$ be the stabilizer of a fixed element $x \in \Omega$. Then there is a bijection $\beta: \Omega \rightarrow G/B$ between $\Omega$ and the left cosets of $B$ where $\beta(y) = gB$ if and only if $gx = y$. If $\{g_1,\dots,g_m\}$ is a set of double coset representatives and we identify a pair $(y,z) \in \Omega \times \Omega$ with its image $(\beta(y),\beta(z)) \in G/B \times G/B$, then the relations can be described as 
\[R_i = \{(hB,hg_iB) \,\, | \,\, h \in G \}. \]

Finally, the relations $R_i$ are in bijection with the double cosets $B \backslash G / B$ as the mapping $\Omega \times \Omega \to B \backslash G / B$: $(gB,hB) \to Bg^{-1}hB$ maps relation $R_i$ to $Bg_iB$. In particular, one can see that when $g_i$ is an involution, the relation $R_i$ is symmetric.

To summarize, we can view association schemes ${\cal{A}}(G,\Omega)$ from transitive group actions entirely in the group itself as ${\cal{A}}(G,G/B)$, with the relations in correspondence to the double cosets $B \backslash G / B$, where $B$ is the stabilizer of an arbitrary element in $\Omega$.

\subsection{Actions of classical groups and $BN$-pairs}

In our case, the group $G$ will be a `projective' classical group acting on flags of a corresponding finite geometry. All of these groups are defined over $\mathbb{F}_q$, tacitly assuming that in the case of the projective unitary groups, the prime power $q$ is a square. We record the groups and their geometries below, together with the Cartan notation and the underlying Weyl group, which will be explained later in this section. We will refer to these groups as the classical groups from now on. 

\begin{table}[h!]
	\centering
	\setlength{\tabcolsep}{5mm}
	\def\arraystretch{1.4} 
	\begin{tabular}{l | l | l | l }
		classical group & geometry & Cartan notation & Weyl group \\
		\hline
		$\PGL(n+1,q)$ & projective space & $A_n(q)$ & $A_n$ \\
		$\PGU(2n+1,q)$ & hermitian polar space & $^2A_{2n}(q)$ & $B_n$ \\
		$\PGU(2n,q)$ & hermitian polar space & $^2A_{2n-1}(q)$ & $B_n$ \\
		$\PSp(2n,q)$ & symplectic polar space & $C_n(q)$ & $B_n$ \\
		$\PGO(2n+1,q)$ & parabolic quadric & $B_n(q)$ &  $B_n$ \\
		$\PGO^-(2n+2,q)$ & elliptic quadric & $^2D_{n+1}(q^2)$ & $B_{n}$ \\
		$\PGO^+(2n,q)$ & hyperbolic quadric & $D_n(q)$ & $D_n$ 		
	\end{tabular}
	\caption{The projective classical groups.}
	\label{table:classicalgroups}		
\end{table}


Since we only consider geometries of rank at least three, we can restrict ourselves to $n \geq 3$ in \Cref{table:classicalgroups}, except for the last entry where we assume $n \geq 4$. We should also mention that in the case of $\PGO^+(2n,q)$, we consider the oriflamme complex of the hyperbolic quadric as the geometry, i.e. explicitly making the distinction between the two classes of generators. One could also consider it as a polar space of rank $n$ by not making the distinction. The underlying Weyl group would then have type $B_n$ \cite[Proposition 7.8.9]{BC13}. The left exponent in the Cartan notation of $\PGU(n,q)$ and $\PGO^-(2n,q)$ indicates that these groups are `twisted', while the others are `untwisted'. The precise definition (see \cite{Carter72}) of this notion is not too important for our purposes, but we will use this terminology later on. \\


Each of the above classical groups comes with a transitive action on the flags of their corresponding finite geometries \cite{Tits74}. 
As we saw in the previous section, a transitive action of a finite group on a finite set leads to an association scheme. However, for the action of classical groups, we can say more. The reason is that these are instances of groups with $BN$-pairs. We will recall some properties these groups possess, but refer to \cite{Garrett97} for a precise definition of a group with a $BN$-pair. There it is also explained that groups with $BN$-pairs and buildings are tightly interwoven.

\begin{property}\label{property:BN-pair}
	A group $G$ with a $BN$-pair has two subgroups $B$ and $N$ such that the following properties are satisfied:
	\begin{enumerate}
		\item $B$ and $N$ generate $G$.
		\item $B \cap N = T$ is normal in $N$ and the quotient $W = N/T$ is generated by a set of involutions $S$.
		\item\label{item:Bruhatdecomp} $B \backslash G / B = \sqcup_{w \in W} B\dot{w}B$, where $\dot{w}$ denotes a representative of $w \in W$ in $N$. For the remainder of this paper, we will write $BwB$ instead of $B\dot{w}B$.
		\item\label{item:Weylgroup} $(W,S)$ is a Coxeter system, i.e. $W$ is a group with generators $S$ and relations $(s_is_j)^{m_{ij}} = (s_js_i)^{m_{ij}}$, where $m_{ij} \geq 2$ if $i \neq j$.
		\item\label{item:relationmult} If $\ell:W \rightarrow \mathbb{N}$ denotes the length function on $W$, then for any $w \in W$ and $s \in S$ we have
		\[(BsB)(BwB) \subseteq \begin{cases}
		BswB & \text{if } \ell(sw) > \ell(w) \\
		BswB \cup BwB & \text{if } \ell(sw) < \ell(w).
		\end{cases}\]
		
	\end{enumerate}
The double coset decomposition in (\labelcref{item:Bruhatdecomp}) is called the Bruhat decomposition of $G$, and the group $W$ in (\labelcref{item:Weylgroup}) is called the Weyl group of $G$. The rank of $W$ is defined as $|S|$. The length function $\ell$ in (\labelcref{item:relationmult}) returns for every $w \in W$ its minimal length as a word in the generators $S$. Although a single element could be given by different words, it follows more or less from the fact that the relation $(s_is_j)^{m_{ij}} = (s_js_i)^{m_{ij}}$ has equally many factors on both sides that this function is well-defined.
\end{property}

It is well known that a presentaion of the Weyl groups $A_n$, $B_n$ and $D_n$ appearing in \Cref{table:classicalgroups} can be given by their Dynkin diagrams. This works as follows: a node with label $i$ corresponds to a generator $s_i \in S$. The relations are $s_i^2=1$ and $(s_is_j)^{m_{ij}}=1$ for $i \neq j$,  where $m_{ij}-2$ is the number of lines connecting $s_i$ and $s_j$. 

Since the classical groups are finite, each of the corresponding Weyl groups is finite (which is the defining property of a spherical building) and there is a unique longest word with respect to the length function defined on $W$. This element is commonly denoted by $w_0$ and its length $\ell(w_0)$ is recorded in the last column in the table. We remark that $\ell(w_0)$ as stated is only valid for the untwisted groups. In the Weyl group of the twisted groups, the length function is slightly modified.

These Dynkin diagrams will be of great use later on, as they convey more information than just the presentation of the Weyl group, we will return to this in \Cref{section:opposition}. 

\begin{table}[h!]
	\centering
	\setlength{\tabcolsep}{8mm}
	\def\arraystretch{4} 
	\begin{tabular}{c | c | c}
		Weyl group & Dynkin diagram & $\ell(w_0)$ \\
		\hline
		$A_n$ & \begin{tikzpicture}[very thick, baseline={(N1)},scale=0.8]
		\def\a{1}
		\tikzset{dynkin/.style={circle,draw,minimum size=0.5mm,fill}}
		\path
		(0,0)      node[dynkin] (N1) {} 
		+(90:.5) node{$1$}
		
		++(0:2*\a) node[dynkin] (N2) {} 
		+(90:.5) node{$2$}
		
		++(0:\a)   coordinate (A) ++(0:\a) coordinate (B)
		++(0:\a) node[dynkin] (N3) {}
		+(90:.5) node{${n-1}$}
		
		++(0:2*\a) node[dynkin] (N4) {} 
		+(90:.5) node{$n$};
		
		\draw[dashed] (A)--(B);
		\draw (N1)--(N2)--(A) (B)--(N3)--(N4);
		\end{tikzpicture} & $\dfrac{n(n+1)}{2}$ \\
		
		$B_n$ & \begin{tikzpicture}[very thick,baseline={(N1)},scale=0.8]
		\def\a{1}
		\tikzset{dynkin/.style={circle,draw,minimum size=0.5mm,fill}}
		\path
		(0,0)      node[dynkin] (N1) {} 
		+(90:.5) node{$1$}
		
		++(0:2*\a) node[dynkin] (N2) {} 
		+(90:.5) node{$2$}
		
		++(0:\a)   coordinate (A) ++(0:\a) coordinate (B)
		++(0:\a) node[dynkin] (N3) {}
		+(90:.5) node{${n-1}$}
		
		++(0:2*\a) node[dynkin] (N4) {} 
		+(90:.5) node{$n$};
		
		\draw[dashed] (A)--(B);
		\draw (N1)--(N2)--(A) (B)--(N3);
		\draw[double distance=2pt] (N3)--(N4);
		\end{tikzpicture} & $n^2$ \\
		
		$D_n$ & \begin{tikzpicture}[very thick,baseline={(N1)},scale=0.8]
		\def\a{1}
		\tikzset{dynkin/.style={circle,draw,minimum size=0.5mm,fill}}
		\path
		(0,0)    node[dynkin] (N1) {} 
		+(90:.5) node{$1$}
		
		++(0:2*\a) node[dynkin] (N2) {} 
		+(90:.5) node{$2$}
		
		++(0:\a) coordinate (A) 
		++(0:\a) coordinate (B)
		++(0:\a) node[dynkin] (N3) {}
		+(90:.5) node{$n-2$}
		
		++(20:1.7*\a) node[dynkin] (N4) {} 
		+(90:.5) node{$n-1$}
		
		(N3) ++(-20:1.7*\a) node[dynkin] (N5) {} 
		+(90:.5) node{$n$};
		
		\draw[dashed] (A)--(B);
		\draw (N1)--(N2)--(A) (B)--(N3)--(N4) (N3)--(N5);
		\end{tikzpicture} & $n(n-1)$ \\
	\end{tabular}
	\caption{The Weyl groups $A_n$, $B_n$ and $D_n$.}
	\label{table:weylgroups}
\end{table}

\begin{exa}\label{exa:GLexample1}
	One can consider $G = \PGL(n+1,q)$ of type $A_n$ with $B$ the subgroup of upper triangular matrices and $N$ the subgroup of monomial matrices, i.e. matrices with exactly one non-zero element in every row and column. Then $T = B \cap N$ is the subgroup of diagonal matrices and one can see that the corresponding Weyl group $N/T$ 
	is isomorphic to the symmetric group $\Sym(n+1)$. In fact, we will identify the Weyl group of type $A_n$ with $\Sym(n+1)$, with the adjacent transposition  $s_i = (i, i+1)$, $1 \leq i \leq n-1$, as the generating involutions.	
\end{exa}

In this example, we can explicitly infer the action of the Weyl group on the geometry. This is harder to do in general, but the upshot is that it is in fact not necessary. One of the key ideas from the theory of buildings is that we can investigate the action of the Weyl group on the whole geometry by considering its action on a specific substructure, which we will refer to as a frame. In the example above, a projective frame consists of $n+1$ linearly independent vectors $\{e_1,\dots,e_{n+1}\}$ in the underlying vector space $V(n+1,q)$, and the action of $\Sym(n+1)$ is the natural action on the indices. We will work out this example for $n=2$ in more detail later on in \Cref{exa:GLexample2}.

Let's return to the association scheme ${\cal{A}}(G,G/B)$, where $G$ is a classical group as in \Cref{table:classicalgroups} and $B$ the stabilizer of a maximal flag. In fact, through the correspondence of buildings and groups with $BN$-pairs \cite[Chapter 5]{Garrett97}, this subgroup $B$ is also one of two parts of a $BN$-pair. The subgroup $N$ on the other hand is the stabilizer of an apartment (i.e. a thin subgeometry of the same type \cite[Definition 6.5.3]{BC13}), in the language of buildings. We saw that the double cosets $B \backslash G / B$ correspond to the relations in the scheme ${\cal{A}}(G,G/B)$. But now we know more:
\begin{enumerate}
	\item The relations in ${\cal{A}}(G,G/B)$ are indexed by the underlying Coxeter group $W$. We can therefore write $\{R_w \,\, | \,\, w \in W\}$ instead of $\{R_1,\dots,R_m\}$.
	\item If $(x,y) \in R_{w_1}$ and $(y,z) \in R_{w_2}$, we can determine which relations are possible for $(x,z)$ by (\labelcref{item:relationmult}).
\end{enumerate}

\begin{exa}\label{exa:GLexample2}
	Although we usually assume $n \geq 3$ for type $A_n$, retake \Cref{exa:GLexample1} with $n = 2$ for the sake of simplicity. We have the action of $\PGL(3,q)$ on the subspaces of $\PG(2,q)$. The subgroup $B$ can be seen as the stabilizer of the standard flag $\{\langle e_1 \rangle, \langle e_1,e_2 \rangle, \langle e_1,e_2,e_3 \rangle\}$, where $e_i$ is the $i$-th basis unit vector. The last element of every maximal flag is $\PG(2,q)$ itself, so we can omit this and talk about \{point, line\}-flags. 
	
	The underlying Coxeter group is $\Sym(3)$ and hence it follows that there are $|\Sym(3)|=6$ relations on \{point, line\}-flags in $\PG(2,q)$, which can be seen in the table below. To figure out what the geometrical interpretation of the double coset $BwB$ with $w \in \Sym(3)$ is, one needs to look at the action of $w$ on $e_i$ defined by $we_i = e_{w(i)}$. For example, $s_1s_2 = (1\,\, 2 \,\, 3)$ maps the standard flag to $\{\langle e_2 \rangle, \langle e_2,e_3 \rangle, \langle e_1,e_2,e_3 \rangle\}$, which means that the relation between two flags corresponding to $Bs_1s_2B$ is the relation where the point of the second flag is on the line of the first but not vice versa. This already shows that ${\cal{A}}(G,G/B)$ is not symmetric.
	
\begin{table}[h]
	\begin{center}
		\begin{tabular}{c | c }	
		double coset	& geometric interpretation	\\
		\hline
		$B1B$			& $p_1=p_2,\ell_1=\ell_2$  \\
		$Bs_1B$			& $p_1 \neq p_2,\ell_1=\ell_2$  \\
		$Bs_2B$			& $p_1=p_2,\ell_1 \neq \ell_2$  \\
		$Bs_1s_2B$		& $p_1 \notin \ell_2, p_2 \in \ell_1$ \\
		$Bs_2s_1B$		& $p_1 \in \ell_2, p_2 \notin \ell_1$ \\
		$Bs_1s_2s_1B$	& $p_1 \notin \ell_2, p_2 \notin \ell_1$
		\end{tabular}
	\caption{The 6 possibilities for the mutual position of two flags $\{p_1,\ell_1\}$ and $\{p_2,\ell_2\}$ in $\PG(2,q)$. Remark that $s_1s_2s_1 = s_2s_1s_2$.}
	\label{table:example}
	\end{center}
\end{table}

One can now check, both algebraically by (\labelcref{item:relationmult}) in \Cref{property:BN-pair} and geometrically, that if $(\{p_1,\ell_1\},\{p_2,\ell_2\}) \in R_{s_1}$ and $(\{p_2,\ell_2\},\{p_3,\ell_3\}) \in R_{s_2}$ then $(\{p_1,\ell_1\},\{p_3,\ell_3\}) \in R_{s_1s_2}$.

\end{exa}

\subsection{Iwahori-Hecke algebras}

For an association scheme we have a corresponding Bose-Mesner algebra, given by basis matrices $A_1,\dots,A_m$. We know by definition that the product $A_iA_j$ can again be expressed as a linear combination in the basis matrices: 
\[A_iA_j \in \langle A_1,\dots,A_m \rangle.\]
This information is encoded in the complex product of two matrices $A_i$ and $A_j$, as defined in \cite{Zieschang96}, which returns the basis matrices $A_k$ appearing in this linear combination with non-zero coefficients. This provides us some qualitative information on how the basis matrices are related. However, we are often, if not always, interested in the coefficients $p_{ij}^k$, also called the intersection numbers:
\[A_iA_j =\sum_{k=1}^m p_{ij}^kA_k.\]

When we look back at the previous section, we see that in fact (\labelcref{item:relationmult}) in \Cref{property:BN-pair} corresponds to the complex multiplication of basis relations in the association scheme $\mathcal{A}(G,G/B)$, where $G$ is a classical group with a $BN$-pair. We now would like to make this information more precise, by specifying the intersection numbers. This is the content of the next result, due to Iwahori and Matsumoto. We refer to \cite[Section 8.4]{GP00} in which this result and others can be found, along with more historical context.

\begin{prop}
	Let $\{A_w \,\, | \,\, w \in W\}$ be the adjacency matrices of the relations $R_w$ defined in the previous section for a group $G$ with a $BN$-pair and underlying Coxeter group $W$. Then we have the following multiplication rules for all $s \in S$ and $w \in W$:	
	
	\[A_sA_w = \begin{cases}
	A_{sw} & \text{if } \ell(sw) > \ell(w) \\
	q_sA_{sw}+(q_s-1)A_w & \text{if } \ell(sw) < \ell(w),
	\end{cases}\]
	where $q_s = |BsB/B|$ for all $s \in S$.
\end{prop}

The set $BsB/B$ is the set of left cosets contained in $BsB$, which can also be seen as the set of left cosets in relation $R_s$ with $B$.
It is known that $q_s = q_t$ whenever $s$ and $t$ are conjugate in $W$.
Moreover, $q_s = q^{c_s}$, for some $c_s \in \mathbb{N}$, where $q$ is the order of the underlying field of the classical group. When the group is untwisted $q_s = q$ for all $s \in S$.

\begin{exa}\label{exa:structureconstantA2}
	For the earlier example of $\PGL(3,q)$ we find $q_s = |BsB/B| = q$ for all $s \in S$. This number is also the valency of the relation $R_s$, which can be seen from \Cref{table:example}.
\end{exa}

\begin{remark}\label{remark:ihalgebraasqdeformation}
	If we replace $q_s$ by 1 in the multiplication rules above, we would obtain an algebra that is isomorphic to $\C W$. This is no coincidence, as both the association scheme and the group algebra $\C W$ can be seen as specializations of a more general algebra, which is a generic Iwahori-Hecke algebra \cite[Section 8.1]{GP00}. Intuitively, one can see the association scheme $\mathcal{A}(G,G/B)$ as a deformation or $q$-analog of $\C W$. 
\end{remark}

The connection between the association scheme for flags of incidence geometries and Iwahori-Hecke algebras has been pointed out before. The earliest reference we could find is due to Ott \cite[p108]{Ott81}. For a more recent account, we refer to \cite{Guillot20}, where it is reproven that if we remove groups from the equation and focus on the combinatorial side, the adjacency algebra for buildings obtained is the Iwahori-Hecke algebra as above. In summary, the main advantage of this connection is that the foundations for the representation theory of Iwahori-Hecke algebras has been studied intensively in the last 50 years and is immediately applicable to the combinatorial problem we consider.

\subsection{Opposition in spherical buildings}\label{section:opposition}

The connections we made in the previous sections will now be put to use. Let $G$ be one of the classical groups, $W$ its underlying Weyl group of rank $n$ and $w_0$ the longest word in $W$. 

\begin{defin}\label{def:opposition}
	The relation $R_{w_0}$ in $\mathcal{A}(G,G/B)$ is called the \textbf{opposition relation}.
\end{defin} 

Remark that $w_0$ is an involution (as $w_0^{-1}$ has the same length as $w_0$). Therefore, $w_0 = w_0^{-1}$ and we conclude that opposition is a symmetric relation. Although this definition is made in terms of the group $G$, it in fact corresponds to a geometrical notion of `far awayness' of maximal flags of which one can see an example in \Cref{table:example}. 

This definition can be extended to partial flags. To do so, we need the notion of type and cotype of a partial flag.

\begin{defin}\label{def:type}
	Let $\Gamma$ be the geometry corresponding to $G$. The \textbf{type of a subspace} in $\Gamma$ is its vector space dimension. The \textbf{type of a flag} is the set of types of the subspaces appearing in the flag. The \textbf{cotype of a flag} is its complement in $[n]:=\{1,\dots,n\}$.
\end{defin}

Whenever we use the term `dimension' in the remainder of this paper, we refer to the vector space dimension so that it coincides with the type. The type or cotype of a partial flag can also be defined in terms of $G$ and its Weyl group $W$.

\begin{defin}
	For every $J \subseteq [n]$, we can define a \textbf{parabolic subgroup} $P_J$ of $G$ by $P_J := BW_JB$, where $W_J := \langle s_i \,\, | \,\, i \in J \rangle$.
\end{defin}

It turns out that $P_J$ is the stabiliser of a flag of cotype $J$ \cite{Garrett97}. In fact, it is the stabiliser of the unique flag of cotype $J$ contained in the maximal flag stabilised by $B$. In other words, we can find a bijection between flags of cotype $J$ and $G/P_J$. Remark that when $J = \emptyset$, we retrieve the original bijection between maximal flags and cosets in $G/B$.

This bijection allows us to identify the cotype $J \subseteq [n]$ of a flag with a set of generators $\{s_i \,\, | \,\, i \in J\} \subseteq S$. We will therefore refer to the latter set as the cotype of the flag as well. In this way, since $w_0$ acts on $S$ by conjugation \cite[Lemma 1.5.3]{GP00}, we can say that $w_0$ acts on types as well. Furthermore, this action can most easily be seen by looking at the Dynkin diagram as generators correspond to nodes. In the case of $A_n$ and $D_n$, $n$ odd, the action of $w_0$ corresponds to the unique diagram automorphism of order 2. In the case of $B_n$ and $D_n$, $n$ even, it is the identity.

\begin{defin}\label{def:oppositionofflags}
	Two partial flags of cotypes $J$ and $K$ are \textbf{opposite} if and only if $J^{w_0} = K$ and they can be extended to two maximal opposite flags.
\end{defin}

Alternatively, we can define oppositeness of two partial flags of cotypes $J$ and $K$ as in \cite{Brouwer10} by first identifying them as two cosets $gP_J$ and $hP_K$ and then requiring that $P_Kh^{-1}gP_J = P_Kw_0P_J$. 

\begin{exa}
	Retake the running example $G = \PGL(3,q)$. From our previous calculations, we see that $s_1$ stabilises the line $\langle e_1,e_2 \rangle$, which is a flag of cotype $\{1\}$ or $\{s_1\}$, and $s_2$ stabilizes the point $\langle e_1 \rangle$, a flag of cotype $\{2\}$ or $\{s_2\}$. The element $w_0$ in $\Sym(3)$ is $s_1s_2s_1 = (1 \,\, 3)$ and hence $s_1^{w_0} = s_2$ as expected. From \Cref{table:example}, we deduce that two maximal flags are opposite if and only if the point (resp. the line) of the first is not incident with the line (resp. the point) of the second. For partial flags we see that an element of cotype $\{s_1\}$, i.e. a line, can only be opposite to an element of cotype $\{s_2\}$, i.e. a point. This happens whenever the point and line are not incident.
\end{exa}  

Although oppositeness of flags is defined as a global property, the previous example shows that checking whether two given \{point, line\}-flags in $\PG(2,q)$ are opposite can be done locally, or `elementwise', by running over the elements of the first flag and checking incidence with the element of opposite type of the second flag. This is an instance of a more general phenomenon, which we can phrase on the level of the groups. This property can also be seen geometrically, but would require a separate discussion for each type of Weyl group and will be made clear in later sections.

\begin{lemma}\label{lemma:oppositenesslocally}
	Let $gB, hB \in G/B$ then the following are equivalent:
	\begin{enumerate}
		\item $Bh^{-1}gB = Bw_0B$,
		\item $P_Jh^{-1}gP_K = P_Jw_0P_K$ for all pairs $(P_J,P_K)$ of parabolic subgroups s.t. $J^{w_0} = K$,
		\item $P_Jh^{-1}gP_K = P_Jw_0P_K$ for all pairs $(P_J,P_K)$ of maximal parabolic subgroups s.t. $J^{w_0} = K$,
	\end{enumerate}
\end{lemma}  

Remark that a maximal parabolic subgroup $P_J$, where $J = S \setminus \{i\}$, corresponds to the stabilizer of a subspace of type $i$ in the geometry. Therefore, we can identify the cosets $gP_J$ with the $i$-dimensional spaces in the corresponding geometry. 

\begin{proof}
	Multiplying the equation $Bh^{-1}gB = Bw_0B$ on the left by $P_J$ and on the right by $P_K$ shows that $1) \Rightarrow 2)$. The implication $2) \Rightarrow 3)$ is immediate.
	So suppose that $3)$ holds. As $B \leq P_J$, for any parabolic subgroup $P_J$ we have $Bh^{-1}gB \subseteq P_Jh^{-1}gP_K$. Moreover, by the definition of the parabolic subgroup, we know that $P_J = BW_JB = \cup_{w \in W_J}BwB$. Therefore, $P_Jh^{-1}gP_K = P_Jw_0P_K = (BW_JB)(Bw_0B)(BW_KB)$ and the latter can be expressed as a union of double cosets by the multiplication rule (\labelcref{item:relationmult}) in \Cref{property:BN-pair}. Combining all of this, we find that for any pair $(P_J,P_K)$ of maximal parabolic subgroups such that $J^{w_0} = K$ the following holds:	
	\[Bh^{-1}gB \subseteq P_Jh^{-1}gP_K = (BW_JB)(Bw_0B)(BW_KB) \subseteq \bigcup_{u \in W_J, v \in W_K}Buw_0vB = \bigcup_{v \in W_K}Bw_0vB,\]
	where the last equality follows as $uw_0v = w_0u^{w_0}v \in w_0W_K$, recalling $J^{w_0} = K$. Since this expression holds for any maximal $K \subset S$, we obtain that $Bh^{-1}gB$ must be contained in the intersection of $\bigcup_{v \in W_K}Bw_0vB$ over all $K$. The intersection consists of all double cosets $Bw_0xB$ such that $x \in W_K$ for all maximal $K \subset S$. However, as $W_{K_1} \cap W_{K_2} = W_{K_1 \cap K_2}$ for all $K_1, K_2 \subseteq S$, it follows that $x=1$ and hence $Bh^{-1}gB \subseteq Bw_0B$, from which equality follows, as double cosets are either disjoint or equal.
\end{proof}

\subsection{Erd\H{o}s-Ko-Rado sets of flags}

The opposition relation, or rather its complement, will serve as the defining one for Erd\H{o}s-Ko-Rado problems in this context.

\begin{defin}
	Let $\Omega$ be a set of flags of cotype $J$ such that $J^{w_0} = J$. Then $C \subseteq \Omega$ is an \textbf{EKR-set of flags} if $R_{w_0} \cap (C \times C) = \emptyset$. In other words, $C$ is an EKR-set of flags if no two flags in $C$ are opposite to each other.
\end{defin}

We have the following important feature of EKR-sets of flags. Consider the projection map $\phi_J: G/B \to G/P_J$ defined by $\phi_J(gB) = gP_J$. Geometrically, this amounts to deleting elements in a maximal flag such that the remainder is a flag is of cotype $J$. In the other way, we can `blow up' a partial flag $gP_J$ of cotype $J$ to a set of maximal flags $\phi^{-1}_J(gP_J)$.

\begin{lemma}\label{lemma:blow-up}
	If $C$ is an EKR-set of flags of cotype $J$, then $\phi^{-1}_J(C)$ is an EKR-set of maximal flags.
\end{lemma}

\begin{proof}
	Suppose that $\phi^{-1}_J(C)$ contains two opposite flags $gB$ and $hB$. Then $Bh^{-1}gB = Bw_0B$, which implies by \Cref{lemma:oppositenesslocally} that $P_Jh^{-1}gP_J = P_Jw_0P_J$, which contradicts our assumption.
\end{proof}

Geometrically this is quite natural as well: partial flags that are not opposite cannot suddenly become opposite when extending them to maximal flags. As said before, we will indicate the geometrical meaning of oppositeness of flags and hence EKR-sets in each classical type in more detail in the following sections. This observation comes in quite handy when we discuss upper bounds for the size of EKR-sets of flags: suppose we have an upper bound in the case of maximal flags, then an upper bound for an EKR-set $C$ of flags of cotype $J$ is derived by blowing up $C$ to $\phi^{-1}_J(C)$ and applying the bound in case of maximal flags to the latter set. Since $|\phi_J^{-1}(C)| = |C||\phi_J^{-1}(B)|$, an upper bound for $|C|$ can then immediately be computed.

A classical tool to obtain upper bounds for EKR-problems is the Delsarte-Hoffman bound for cocliques in a regular graph, see for example \cite{GM16}.

\begin{prop}\label{prop:DelsarteHoffman}
	Let $C$ be a coclique in a $k$-regular graph. Then 
	\[|C| \leq \frac{v}{1-\frac{k}{\alpha}},\]
	where $v$ is the number of vertices and $\alpha$ the smallest eigenvalue of the adjacency matrix of the graph. Moreover, if equality is attained, then the characteristic vector of $C$ is contained in the sum of eigenspaces corresponding to the eigenvalues $k$ and $\alpha$.
\end{prop}

In our case, an EKR-set of maximal flags corresponds to a coclique in the graph with adjacency matrix $A_{w_0}$. This graph is surely regular, as there is a group acting transitively on it. All that remains to do to obtain upper bounds, is to find the eigenvalues of $A_{w_0}$. We recall that in order to do so, it suffices to find the irreducible representations of the Iwahori-Hecke algebra to which it belongs. These irreducible representations and their corresponding characters are well-studied and known results can be applied to obtain the eigenvalues. It turns out that the irreducible characters of the Iwahori-Hecke algebra and those of $W$ are intimately connected, which makes sense in light of \Cref{remark:ihalgebraasqdeformation}. In essence, this is the approach followed by Brouwer in \cite{Brouwer10}.

\subsection{Tying it all together: Brouwer's recipe for the eigenvalues of oppositeness}

With this approach, Brouwer could show that the eigenvalues of oppositeness are powers of $q$, when the group is defined over $\mathbb{F}_q$ \cite{Brouwer10}. From his and earlier work, it is in fact possible to extract an explicit algorithm to compute the eigenvalues of $A_{w_0}$, which we have recorded below. 


\begin{algo}\label{algorithmmaximal}
	Computation of eigenvalues of opposition for maximal flags in the geometry corresponding to the classical group $G$ with Coxeter system $(W,S)$ and structure constants $\{q_s \,\, | \,\, s \in S\}$. For every $\chi \in \mathrm{Irr}(W)$ we can compute the corresponding eigenvalue(s) $\lambda_\chi$ as follows.
	\begin{enumerate}
		\item Determine a set $R \subseteq S$ of representatives for the conjugacy classes of generators $\{s^W \,\, | \,\, s \in S\}$.
		\item Compute the values of $e_r=|r^W|(1+\chi(r)/\chi(1))$ for all $r \in R$.
		\item Determine $\chi(w_0)$ and compare to $\chi(1)$ to determine the sign
		\[\mathrm{sgn}(\chi) = \begin{cases}
		+ & \text{if } \chi(1) = \chi(w_0) \\
		- & \text{if }\chi(1) = -\chi(w_0) \\
		\pm & \text{if } |\chi(1)| \neq |\chi(w_0)|.
		\end{cases}\]
		\item Compute the eigenvalue(s) $\lambda_\chi = \mathrm{sgn}(\chi)\prod_r q_r^{e_r/2}$.
	\end{enumerate}
	
	Remark that both the set of structure constants and the set $R$ in step two will consist of at most two elements, as all finite Weyl groups have at most two conjugacy classes containing generators.
\end{algo}

Without going too much into detail, the essence of this algorithm, and Brouwer's paper by extension, is the following result due to Springer combined with the fact that the representation theory and irreducible characters of the association scheme $\mathcal{A}(G,G/B)$ are strongly related to that of the underlying Weyl group $W$.

\begin{theorem}\cite[Theorem 9.2.2]{GP00}
	Let $\chi$ be an irreducible character of $\mathcal{A}(G,G/B)$. The element $A_{w_0}^2$ is central in this algebra and acts on a simple module affording $\chi$ by the scalar $\prod_{r \in R} q_r^{e_r}$.
\end{theorem}

In the case one looks at partial flags instead of maximal flags, it is possible to derive the eigenvalues, up to sign, from the first algorithm. There is some caution to be taken however. If the longest word $w_0$ is central in $W$, then $|\chi(w_0)| = \chi(1)$ for all irreducible characters $\chi$ and so there will always be a single sign. This happens for the Weyl groups of type $B_n$ and $D_n$, $n$ even, but not for type $A_n, D_n$, $n$ odd. In these cases, there is an extra difficulty in computing the eigenvalues of opposition on partial flags as one needs to pay closer attention to the signs. We will therefore not discuss the latter cases in detail when considering partial flags.  

\begin{algo}\label{algorithmpartial}
	Computation of eigenvalues of opposition for partial flags of cotype $J$ in the geometry corresponding to the classical group $G$ with Coxeter system $(W,S)$ and structure constants $\{q_s \,\, | \,\, s \in S\}$.
	\begin{enumerate}
		\item Determine the decomposition of the induced character into irreducibles, i.e. $\mathrm{ind}^W_{W_J}(1_{W_J}) = \sum \chi$.
		\item Compute the eigenvalue $\lambda_\chi$ for every $\chi$ appearing in the above sum using \Cref{algorithmmaximal}.
		\item Compute the length $\ell$ of the longest word in $W_J$.		
		\item Compute the eigenvalues $\mu_\chi = \lambda_\chi \cdot q^{-\ell}$.
	\end{enumerate}
\end{algo}

In particular we see that if the irreducible characters $\chi$ and $\phi$ appear in the decomposition of $1_{W_J}^W$, the ratio of eigenvalues $\mu_{\chi}/\mu_{\phi} = \lambda_{\chi}/\lambda_{\phi}$ is independent of the flag type, which is very relevant considering the denominator in the Delsarte-Hoffman bound.

\begin{remark}
	It is remarkable that the decomposition of $\mathrm{ind}^W_{W_J}(1_{W_J})$, which is a character in a finite Weyl group, determines the structure of the permutation character obtained from the action of $G$ on $G/P_J$. This result is due to Curtis, Iwahori and Kilmoyer \cite{CIK71}, where the relation between the two is investigated on the level of the Iwahori-Hecke algebras. The importance for our applications is that for a given rank $n$, the building of type $X_n$ typically depends on a prime power $q$, while $W$ is a fixed and Weyl Coxeter group. This implies that the decomposition of $\mathrm{ind}^W_{W_J}(1_{W_J})$ can simply be found by a computer!
\end{remark}

\begin{remark}
	The subgroup $W_J$ is not necessarily an irreducible Weyl group, which means that we cannot use the values in \Cref{table:weylgroups} directly. In general, it will be the direct product of irreducible Weyl groups of type $A$, $B$ or $D$. Combining this with the fact that the length of the longest word in $W_1 \times W_2$ is the sum of the lengths of the longest words in the Weyl groups $W_1$ and $W_2$ allows one to compute the value $\ell$.
\end{remark}

\begin{remark}\label{rem:partialflagslosefactor}
	We can see why the eigenvalues lose a factor $q^\ell$, with $\ell$ the length of the longest word in $W_J$ using the projection map $\phi_J$ defined in \Cref{section:opposition}. Consider the partition of maximal flags into the fibers $\{\phi_J^{-1}(gP_J)\}$. If two partial flags $gP_J$ and $hP_J$ are non-opposite, then we know by \Cref{lemma:blow-up} that any two maximal flags in $\phi_J^{-1}(gP_J)$ and $\phi_J^{-1}(hP_J)$ are also non-opposite. On the other hand, if $gP_J$ and $hP_J$ are opposite, then there are $q^\ell$ maximal flags in $\phi^{-1}(hP_J)$ opposite to a given flag in $\phi^{-1}(gP_J)$ by \cite[Corollary 3.2]{Brouwer10}. In other words, the partition is equitable. If we now denote by $\widetilde{A}_{w_0}$ the adjacency matrix of oppositeness for partial flags, then the quotient matrix of this equitable partition is $q^\ell\widetilde{A}_{w_0}$ and the relation between the eigenvalues follows.
\end{remark}	

\section{Applying the algorithms}


We will now apply \Cref{algorithmmaximal} to find the largest and smallest eigenvalues of $A_{w_0}$, which we can feed into the Delsarte-Hoffman bound to obtain upper bounds for the size of an EKR-set of maximal flags. As mentioned before, we can restrict ourselves at first to maximal flags, as we can blow-up any EKR-set of partial flags by \Cref{lemma:blow-up}. As a next step, one could in theory try to obtain better upper bounds for partial flags by considering the oppositeness graph on partial flags and applying the Delsarte-Hoffman bound there. We will show no improvement is possible in the geometries we consider.




\subsection{Classical groups with Weyl group $A_n$, $n \geq 3$}


\subsubsection{The upper bound}

In \Cref{table:classicalgroups} we saw that $\PGL(n+1,q)$ has underlying Weyl group $A_n$ and corresponding geometry $\PG(n,q)$. Moreover, in \Cref{exa:GLexample1} we saw that $A_n \cong \Sym(n+1)$. This means that in order to apply \Cref{algorithmmaximal}, we need to recall some of the representation theory of $\Sym(n+1)$. This theory is well-known and can be found in several books. We will rely on $\cite{GP00}$ for the necessary results. In the end, the main result of this section will be the following upper bound for EKR-sets of maximal flags.

\begin{theorem}\label{thm:boundtypeAmaximal}
	If $C$ is an EKR-set of maximal flags in $\PG(n,q)$, then
	\[|C| \leq \frac{\qbinom{n+1}{1}\qbinom{n}{1}\dots \qbinom{2}{1}\qbinom{1}{1}}{1+q^{(n+1)/2}}.\]
\end{theorem}

Since the enumerator is the number of maximal flags in $\PG(n,q)$, we need to prove that the denominator in the Delsarte-Hoffman bound $1-k/\alpha$ equals $1+q^{(n+1)/2}$. We will run through each of the steps in \Cref{algorithmmaximal} in order to compute $k$ and $\alpha$. Before diving into it, we will need the following definition, which is fundamental for the representation theory of $\Sym(n)$.

\begin{defin}
	A \textbf{partition} $\mu$ is a non-increasing sequence of positive integers $\mu_1 \geq \dots \geq \mu_k$. It is a partition for $m \in \mathbb{N}$, denoted as $\mu \vdash m$, if $|\mu| := \mu_1 + \dots + \mu_k = m$. A partition $\mu$ is usually written with square brackets as $[\mu_1,\dots,\mu_k]$. The total number of partitions of $m$ is denoted as $p(m)$.
\end{defin}

With this definition in mind, we can run through the steps.

\begin{enumerate}
	\item \textbf{Determine a set $R$ of representatives for the conjugacy classes of generators $\{s^W \,\, | \,\, s \in S\}$.}
	
	Each conjugacy classes of $\Sym(n+1)$ is defined by its cycle type. For instance, each generator $s_i = (i, i+1)$, $1 \leq i \leq n$, has cycle type $(2,1^{n-1})$, which implies that all generators are conjugated. The number of elements of this cycle type is clearly $n(n+1)/2$ and any generator is a representative for the conjugacy class.
	
	\item \textbf{For $\chi \in \mathrm{Irr}(W)$, compute the value of $e_r=|r^W|(1+\chi(r)/\chi(1))$.}
	
	From the previous point, it is clear that the number of conjugacy classes is $p(n+1)$. From representation theory we know that this also equals the number of irreducible characters. A more surprising fact is that for $\Sym(n+1)$, there is a very nice bijection connecting a partition $\mu \vdash n+1$ to an irreducible character $\chi_\mu$. This correspondence allows one to evaluate character values on group elements by some elementary combinatorial rules as we will see later on.
	
	This bijection moreover means that we will not explicitly compute all eigenvalues of $A_{w_0}$, as the number of eigenvalues is $p(n+1) \sim e^{C\sqrt{n}}$. We will focus instead on the largest and smallest eigenvalue by figuring out which characters, or equivalently which partitions, give the largest values of $\chi(r)/\chi(1)$.
	
	\item \textbf{Determine $\chi(w_0)$ and compare to $\chi(1)$ to determine the sign.}
	
	The longest word in $\Sym(n+1)$ is $w_0 = (1, n+1)(2, n)(3, n-1)\dots$. For the characters obtained in the previous step, we will be able to compute $\chi(w_0)$ and deduce the signs.

	\item \textbf{Compute the eigenvalue(s) $\lambda_\chi = \mathrm{sgn}(\chi)q_r^{e_r/2}$.}
	
	The only structure constant is $q_r = |BrB/B| = q$ \cite[Section 8.4]{GP00}. For $n=2$ we saw a geometrical explanation in \Cref{exa:structureconstantA2}, which could be generalized to general $n$. In the end, we will find $k = q^{n(n+1)/2}$ and $\alpha = -q^{(n^2-1)/2}$, which proves \Cref{thm:boundtypeAmaximal}.
\end{enumerate}

The characters maximizing $\chi(r)/\chi(1)$ can be easily found with some more terminology. For a partition $\mu = [\mu_1,\dots,\mu_k]$,  define two invariants $a(\mu)$ and $a^*(\mu)$ as 
\begin{align*}
&a(\mu) = \sum_{i=1}^{k}(i-1)\mu_i, \\
&a^*(\mu) = \frac{1}{2}\sum_{i=1}^{k}\mu_i(\mu_i-1) = \sum_{i=1}^{k}{\mu_i \choose 2}.
\end{align*}
Then the following result will be the basis of our investigation.

\begin{prop}\cite[Proposition 5.4.11]{GP00}\label{prop:ainvariants}
	Let $n \geq 1$ and $\mu \vdash n+1$, then
	\[\frac{n(n+1)}{2}\frac{\chi_\mu(r)}{\chi_\mu(1)} = a^*(\mu)-a(\mu),\]
	where $r$ is a transposition.
\end{prop}

Therefore, the objective, which was in terms of irreducible characters, is now of a purely combinatorial nature: to find the partitions $\mu \vdash n+1$ giving the largest values for $a^*(\mu)-a(\mu)$. 

\begin{lemma}\label{lemma:partitionmaxvalue}
	The two largest values for $a^*(\mu)-a(\mu)$ are attained by $\mu = [n+1]$ and $\mu = [n,1]$. The respective values are 
	\begin{align*}
	a^*([n+1])-a([n+1]) = \frac{n(n+1)}{2}, \\
	a^*([n,1])-a([n,1]) = \frac{n^2-n-2}{2}.
	\end{align*}
\end{lemma}

\begin{proof}
	The two partitions actually have a stronger property: $[n+1]$ and $[n,1]$ attain the two largest values for $a^*(\mu)$ and the two smallest values for $a(\mu)$. The former can be seen as follows: for any partition $\mu = [\mu_1,\dots,\mu_k]$ with all parts at most $n-1$ we find
	\[a^*(\mu) = \frac{1}{2}\sum_{i=1}^{k}\mu_i(\mu_i-1) \leq \frac{1}{2}\sum_{i=1}^{k}\mu_i(n-2) = \frac{(n+1)(n-2)}{2},\]	
	which is smaller than $a^*([n+1]) = n(n+1)/2$ and $a^*([n,1]) = n(n-1)/2$.
	
	The definition of $a(\mu)$ immediately shows the latter part. 
\end{proof}

A more concrete understanding of the irreducible characters $\chi_{[n+1]}$ and $\chi_{[n,1]}$ is in order. Consider the natural action of $\Sym(n+1)$ on $\C^{n+1}$ by $g \cdot e_i = e_{g(i)}$, $1 \leq i \leq n+1$, for $g \in \Sym(n+1)$ and $e_i$ the $i$-th standard unit vector. This action defines a representation of $\Sym(n+1)$ in $\GL(n+1,\C)$ and hence a character $\chi$. This character is not irreducible as the line spanned by the vector $\textbf{j} = \sum_{i=1}^{n+1} e_i$ is an invariant subspace under this action. However, it is not hard to check that $\textbf{j}$ and $V = \textbf{j}^\perp$ are indecomposable with irreducible characters $\chi_\textbf{j}$ and $\chi_V$ respectively, so that $\chi = \chi_\textbf{j} + \chi_V$. Since $g \cdot \textbf{j} = \textbf{j}$ for all $g \in \Sym(n+1)$, it follows that $\chi_\textbf{j}$ is the trivial character. The representation corresponding to $V$ is called the standard representation. As $\chi(g) = |\text{Fix}(g)|$, we then find that $\chi_V(g) = |\text{Fix}(g)|-1$. It turns out, after a bit more work \cite[Proposition 5.4.12]{GP00}, that $\chi_\textbf{j} = \chi_{[n+1]}$ and $\chi_V = \chi_{[n,1]}$.

\begin{lemma}
	The eigenvalues corresponding to $\chi_{[n+1]}$ and $\chi_{[n,1]}$ are $\lambda_{[n+1]} = q^{n(n+1)/2}$ and $\lambda_{[n,1]} = \pm q^{(n^2-1)/2}$ respectively, which are the largest and smallest eigenvalue of $A_{w_0}$.
\end{lemma}
\begin{proof}
	For $\chi_{[n+1]}$ we know that $e_r = |r^W| + a^*([n+1]) - a([n+1]) = n(n+1)$, by \Cref{prop:ainvariants} and $|r^W| = n(n+1)/2$. Since it is the trivial character, we also have $\chi_{[n+1]}(w_0) = 1$ so that in the end we find $\lambda_{[n+1]} = q^{e_r/2}$ as stated.
	For $\chi_{[n,1]}$ on the other hand, we have $e_r = n^2-1$. We observed that $\chi_{[n,1]}(w_0) = \text{Fix}(w_0) - 1$, which is $0$ or $-1$, depending on the parity of $n$, and hence never equal in absolute value to $\chi_{[n,1]}(1) = n$.
\end{proof}

This completes the proof of \Cref{thm:boundtypeAmaximal}. It is no coincidence that the exponent of the largest eigenvalue equals $\ell(w_0)$, see \cite[Proposition 3.1]{Brouwer10}.

We can now wonder if this upper bound can be improved for partial flags of cotype $J$, with $J^{w_0} = J$. Recall that an upper bound for these flags can be obtained from \Cref{thm:boundtypeAmaximal} by blow-up. However, it could happen that the ratio of the largest and smallest eigenvalue is bigger for the oppositeness graph on partial flags. We will show that this is not the case. 

\begin{theorem}\label{thm:boundtypeApartial}
	If $C$ is an EKR-set of flags of cotype $J = J^{w_0}$ in $\PG(n,q)$, then
	\[|C| \leq \frac{v}{1+q^{(n+1)/2}},\]
	where $v = [G:P_J]$ is the total number of flags of cotype $J$.
\end{theorem}

\begin{proof}
We will show that the ratio of the largest and the smallest eigenvalue remains unchanged for partial flags of cotype $J$, by proving that $\chi_{[n+1]}$ and $\chi_{[n,1]}$ appear with non-zero multiplicity in $\mathrm{ind}_{W_J}^W(1_{W_J})$ for every $J \subsetneq S$. This implies by \Cref{algorithmpartial} that both eigenvalues survive upon restriction from maximal to partial flags. This is essentially contained in \cite[Remark 6.3.7]{GP00}, but we include a short proof nevertheless. It relies on the fact that the multiplicity of an irreducible character $\chi$ in the decomposition of a reducible character $\zeta$ equals the inner product 

\[\langle \zeta, \chi \rangle_W = \frac{1}{|W|}\sum_{w \in W}\zeta(g)\overline{\chi(g)}.\]

Moreover, in our case we can use Frobenius reciprocity which says that

\[\langle \mathrm{ind}_{W_J}^W(1_{W_J}), \chi \rangle_W = \langle 1_{W_J}, \chi_{|W_J} \rangle_{W_J},\]
where $\langle \cdot, \cdot \rangle_{W_J}$ denotes the inner product in $W_J$ and $\chi_{|W_J}$ the restriction of $\chi$ to $W_J$.

As $\chi_{[n+1]}$ is the trivial character, its restriction is the trivial character $1_{W_J}$. It follows immediately that its multiplicity is $\langle 1_{W_J},1_{W_J} \rangle_{W_J} = 1$. As for $\chi_{[n,1]}$, it is not hard to see that $W_J = \Sym(a_1+1) \times \dots \times \Sym(a_k+1)$ for some positive integers $a_1,\dots,a_k \geq 1$ with $a_1 + \dots + a_k < n$. Under the action of $W_J$, the module $V = \textbf{j}^\perp \cong \C^n$ corresponding to the standard representation decomposes as the direct sum $\C^{a_1} \oplus \dots \oplus \C^{a_k} \oplus U$, of $W_J$-invariant modules, where the action on $U$ is the trivial one. Since $\dim(U) = n - (a_1+\dots+a_k) >0$, it follows that the trivial character appears as a constituent with non-zero multiplicity in $\chi_{[n,1]|W_J}$. 
\end{proof}

\subsubsection{Reaching the upper bound}

Before we dive into examples that reach the upper bound obtained in the previous section, we will figure out what the geometrical interpretation is of the group-theoretical notion of `oppositeness'. By \Cref{lemma:oppositenesslocally}, it suffices to investigate what opposition means for a single subspace. Since the action $w_0$ switches the types, a $k$-dimensional space, $1 \leq k \leq n$, in $\PG(n,q)$ can only be opposite to a $(n-k+1)$-dimensional space (where dimension is the vector space dimension). The obvious opposition relation for a $k$-dimensional space $U$ and a $(n-k+1)$-dimensional space $V$ would be to impose $U \cap V = \emptyset$ or equivalently $\langle U, V \rangle = \PG(n,q)$. It turns out that this natural geometrical definition corresponds to the group-theoretical one.

\begin{lemma}\label{lem:oppositiongeometricallytypeA}
	A $k$-dimensional space $U$ and an $(n-k+1)$-dimensional space $V$ in $\PG(n,q)$ are opposite as in \Cref{def:oppositionofflags} if and only if $U \cap V = \emptyset$.
\end{lemma}

\begin{proof}
	To see this, we will heavily rely on the correspondence between cosets in $G = \PGL(n+1,q)$ and flags or subspaces (which are flags of size one), as explained in \Cref{section:flagstocosets}. Recall the action of $\Sym(n+1)$ on a projective frame $\{e_1,\dots,e_{n+1}\}$ defined by $w \cdot e_i = e_{w(i)}$ as in \Cref{exa:GLexample2}. Without loss of generality, one can take the $k$-dimensional space $U = \langle e_1,\dots,e_k \rangle$ which hence corresponds to $P_J$, $J = S \setminus \{s_k\}$. Then the opposite parabolic subgroup $P_K$, $K = J^{w_0} = S \setminus \{s_{n-k+1}\}$, corresponds to the $(n-k+1)$-dimensional space $V = \langle e_1,\dots,e_{n-k+1} \rangle$. The coset $w_0P_K$ is clearly opposite to $P_J$ so that the corresponding subspaces must be too. The former is the $(n-k+1)$-dimensional space $w_0 \cdot V := \langle w_0 \cdot e_1, \dots, w_0 \cdot e_{n-k+1} \rangle$. As $w_0 = (1, n+1)(2, n)(3, n-1)\dots$, it follows that $w_0 \cdot V = \langle e_{n+1}, e_{n}, \dots, e_{k+1} \rangle$, which indeed shows $U \cap (w_0 \cdot V) = \emptyset$. Finally, every coset opposite to $P_J$ is of the form $gw_0P_K$, $g \in P_J$. It is known that the stabilizer of a $k$-dimensional space is transitive on the $(n-k+1)$-dimensional spaces not intersecting it, so that we find a bijection between all cosets $gw_0P_K$ opposite to $P_J$ and all $(n-k+1)$-dimensional spaces $V^g$ complementary to a fixed $k$-dimensional space $U$.
\end{proof}

It seems that in order to obtain equality in \Cref{thm:boundtypeApartial}, one needs to have rank $n = 2m-1$ and $m \in J$, $m \geq 2$. There are a few results in the literature that point in this direction.

\begin{itemize}
	\item When $n = 4$, it was shown in \cite{BB17} that a maximal set of EKR-flags of type $\{2,3\}$ has size of order $q^5$. Our bound gives an upper bound of order $q^{11/2}$. 
	\item When $n = 6$, it was shown in \cite{MW19} that a maximal set of EKR-flags of type $\{3,4\}$ has size of order $q^{11}$. Our bound again falls an order of $q^{1/2}$ short and gives an upper bound of order $q^{23/2}$.
	\item When $n \geq 2$, it was shown in \cite{BBG14} that a maximal set of EKR-flags of type $\{1,n\}$ has size of order $(n-1)q^{n-2}$. Our bound is quite a bit off and is of order $q^{(3n-3)/2}$.
	\item When $n = 2m-1$ and the type is $\{m\}$, we obtain the same sharp upper bound as in \cite{FW86}, which is the classical EKR-theorem for $m$-dimensional spaces in a $2m$-dimensional vector space.
\end{itemize}

In \cite{Newman04}, it was shown that in the last case, sharpness only arises whenever $C$ is the set of all $m$-spaces through a fixed point or dually, the set of all $m$-spaces contained in a fixed hyperplane. We can use this construction to obtain sharp constructions in more general types.

\begin{theorem}\label{thm:sharpexamplestypeA}
	Let $n = 2m-1$ and consider flags of type $J$, such that $m \in J$ and $J = J^{w_0}$. If $C$ is the set of all flags of type $J$ such that either
	\begin{enumerate}
		\item [(i)] there is a fixed point such that the $m$-dimensional space of every flag in $C$ contains that point, or
		\item [(ii)] there is a fixed hyperplane such that the $m$-dimensional space of every flag in $C$ is contained in that hyperplane,
	\end{enumerate}
	then $C$ is an EKR-set of flags of type $J$ meeting the upper bound in \Cref{thm:boundtypeApartial}.
\end{theorem}
\begin{proof}
	This can be seen as $C$ is a blow-up of the maximal examples mentioned before the lemma. Similarly as in \Cref{lemma:blow-up}, one can see that $C$ is then again an EKR-set of flags. Moreover, its size meets the upper bound in \Cref{thm:boundtypeApartial} since the original set meets the upper bound and the process of blow-up multiplies both the total number of flags and the number of flags in the EKR-set by the same factor, so that equality is preserved.
\end{proof}

\begin{prob}
	Is \Cref{thm:boundtypeApartial} only sharp for flags of type $J$ whenever $n = 2m-1$ and $m \in J$? 
\end{prob}

\subsection{Classical groups with Weyl group $B_n$, $n \geq 3$}

Next we consider the classical groups with underlying Weyl groups of type $B_n$. These groups have polar spaces as corresponding geometries, which depend on a parameter $e$ as follows. 

\begin{prop}\label{prop:numberofspacestypeB}
	The number of subspaces of type $k$ is
	\[\qbinom{n}{k}\prod_{i=1}^{k}(q^{n+e-i}+1),\]
	where $e = \{0,\frac{1}{2},1,\frac{3}{2},2\}$ is defined as:
	\begin{itemize}
		\item $e = 0$ for $G = \PGO^+(2n,q)$,
		\item $e = \frac{1}{2}$ for $G = \PGU(2n,q)$,
		\item $e = 1$ for $G = \PSp(2n,q)$ or $\PGO(2n+1,q)$,
		\item $e = \frac{3}{2}$ for $G = \PGU(2n+1,q)$,
		\item $e = 2$ for $G = \PGO^-(2n+2,q)$.
	\end{itemize}
\end{prop}	

Equivalently, a subspace of type $n-1$ is incident with $q^e+1$ subspace of type $n$, also called generators. The parameter $e$ will be relevant for the Iwahori-Hecke algebras and their structure constants later on.

We now briefly describe the Weyl group $W$ of type $B_n$ and some properties which we need later on. Most of this material can be found in the first chapter of \cite{GP00}, albeit with small notational differences.

The Weyl group $W$ of type $B_n$ is called the hyperoctahedral group and has size $2^nn!$. It can be constructed as a permutation group on the set $\{-n,,\dots,-1,1,\dots,n\}$ with generators $s_i = (i,i+1)(-i,-i-1))$, $1 \leq i \leq n-1$ and $t = (-n, n)$. On the Dynkin diagram in \Cref{table:weylgroups}, the generators $s_i$ correspond to the first $n-1$ nodes, while $t$ is the last node. Clearly $\Sym(n)$ is a subgroup generated by $\{s_1,\dots,s_{n-1}\}$, which can also be seen from the Dynkin diagram. In particular, all $s_i$ are conjugate to each other and contained in a class of size $n(n-1)$, while the other generator $t$ is in a separate conjugacy class of size $n$.
The longest word in $W$ is $w_0 = (-n, n)\dots(-1,1)$ and is contained in the center of $W$ as it commutes with all generators. As we will see, this makes the computations in \Cref{algorithmmaximal} and \Cref{algorithmpartial} considerably easier than for type $A_n$.

\subsubsection{The upper bound}

Similarly as in type $A_n$, we will go through the steps of \Cref{algorithmmaximal} and prove the following result.

\begin{theorem}\label{thm:boundtypeBmaximal}
	Let $C$ be an EKR-set of maximal flags in a polar space of rank $n$ with parameter $e$. If $e \geq 1$ or $n$ even, then
	\[|C| \leq \frac{\prod_{i=1}^{n}(q^{n+e-i}+1)\qbinom{n}{1}\dots \qbinom{1}{1}}{1+q^{n+e-1}} = \prod_{i=2}^{n}(q^{n+e-i}+1)\qbinom{n}{1}\dots \qbinom{1}{1}.\]
	
	When $e = \frac{1}{2}$ and $n$ is odd, the upper bound is
	\[|C| \leq \frac{\prod_{i=1}^{n}(q^{n+1/2-i}+1)\qbinom{n}{1}\dots \qbinom{1}{1}}{1+q^{n/2}} = \prod_{\substack{ i=1 \\ i \neq (n+1)/2}}^{n}(q^{n+1/2-i}+1)\qbinom{n}{1}\dots \qbinom{1}{1}.\]
	
	When $e = 0$ and $n$ is odd, the upper bound is
	\[|C| \leq \frac{\prod_{i=1}^{n}(q^{n-i}+1)\qbinom{n}{1}\dots \qbinom{1}{1}}{2} = \prod_{i=1}^{n-1}(q^{n-i}+1)\qbinom{n}{1}\dots \qbinom{1}{1}.\]
\end{theorem}

Again, the enumerator is the number of maximal flags in a polar space of rank $n$ and parameter $e$, so that the computation of the smallest and largest eigenvalue will suffice to complete the proof.

\begin{enumerate}
	\item \textbf{Determine a set $R$ of representatives for the conjugacy classes of generators $\{s^W \,\, | \,\, s \in S\}$.}
	
	As mentioned in the introduction of this section, we can take $s: = s_1$ and $t$ to be two representatives of the conjugacy classes, with sizes $|s^W|=n(n-1)$ and $|t^W| = n$ respectively.
	
	\item \textbf{For $\chi \in \mathrm{Irr}(W)$, compute the value of $e_r=|r^W|(1+\chi(r)/\chi(1))$ for all $r \in R = \{s,t\}$.}
	
	The irreducible characters of $W$ are indexed by pairs of partitions $(\mu,\nu)$ such that $|\mu|+|\nu|=n$. Remark that it is allowed for $\mu$ or $\nu$ to be the empty partition $\emptyset$. Again, we will not compute all values, as there are far too many, but focus on the ones giving the smallest and largest eigenvalue. 
	
	For any irreducible character $\chi_{(\mu,\nu)}$ of $W$, or equivalently its group algebra $\C W$, there exists a central character $\omega_{(\mu,\nu)}$, which is a character of the center $Z(\C W)$ of the group algebra \cite[Section 6.2.1]{GP00}. A basis of $Z(\C W)$ is $\{\widehat{\mathcal{C}} = \sum_{g \in \mathcal{C}}g \,\, | \,\, \mathcal{C} \text{ a conjugacy class}\}$ and the evaluation of a central character can be computed as
	\[\omega_{(\mu,\nu)}(\widehat{\mathcal{C}}) = |\mathcal{C}|\frac{\chi_{(\mu,\nu)}(g)}{\chi_{(\mu,\nu)}(1)} \hspace{.5cm} \text{ for } g \in \mathcal{C}.\]
	
	This ratio is exactly the one we are interested in for the eigenvalue computation (see step 2 in \Cref{algorithmmaximal}) and the advantage is that combinatorial formulae exist to evaluate central characters. The same principle underlies the methodology in type $A_n$, although we did not explicitly mention this.
	
	\item \textbf{Determine $\chi(w_0)$ and compare to $\chi(1)$ to determine the sign.}
	
	As $w_0$ is central in $W$, we always have $|\chi(w_0)| = \chi(1)$ for any character $\chi$, which is much easier than type $A_n$. Determining the sign is a simple matter as \cite[Proposition 4.8]{Ram97} tells us that $\chi_{(\mu,\nu)}(w_0) = (-1)^{|\nu|}\chi_{(\mu,\nu)}(1)$.
	
	\item \textbf{Compute the eigenvalue $\lambda_\chi = \mathrm{sgn}(\chi)q_s^{e_s/2}q_t^{e_t/2}$.}
	
	The only thing that remains is to compute the structure constants $q_s$ and $q_t$. There is a combinatorial way to find them, which can also be adapted to type $A_n$.
	
	A polar frame in a polar space of rank $n$ is a set of $2n$ points $\{e_{-n},\dots,e_{-1},e_1,\dots,e_n\}$ such that $e_i$ is collinear to all points except $e_{-i}$ for all $-n \leq i \leq n$, see \cite[Definition 10.4.1]{BC13}. It may be clear from the construction of $W$ as a permutation group that there is an action of $W$ on any polar frame, respecting the underlying geometry. Given a polar frame, we can consider the maximal flag $F =\{\langle e_1 \rangle, \langle e_1,e_2 \rangle, \dots, \langle e_1,\dots,e_n \rangle \}$ and identify it with its stabilizer $B$. Then $sB = s_1B$ corresponds to the flag $\{\langle e_2 \rangle, \langle e_1,e_2 \rangle, \dots, \langle e_1,\dots,e_n \rangle \}$, i.e. a flag that has every space of type $i$ in common with $F$ except for $i=1$. The double coset $BsB$ is the image of this flag under the stabilizer $B$ and hence contains all flags with that intersection property. As there are $q$ ways to change the point $\langle e_1 \rangle$ on the line $\langle e_1, e_2 \rangle$ to a different one, it follows that the number $|BsB/B|$ of maximal flags in $BsB$ equals $q$.
	
	In the same vein is the double coset $BtB$ the set of all maximal flags that have all spaces in common with $F$, except that of type $n$. There are $q^e$ ways to choose a different generator through the subspace of type $n-1$ and hence $|BtB/B| = q^e$.
\end{enumerate}

All that remains is to complete the second step. Recall the invariants $a(\mu) = \sum_{i=1}^{k}(i-1)\mu_i$ and $a^*(\mu) = \sum_{i=1}^{k}{\mu_i \choose 2}$ for a partition $\mu = [\mu_1,\dots,\mu_k]$.

\begin{prop}\cite[Propositions 6.2.6 and 6.2.8]{GP00}
	Let $\chi_{(\mu,\nu)} \in \mathrm{Irr}(W)$, then we have, with $\widehat{s} = \widehat{s^W}$ and $\widehat{t} = \widehat{t^W}$,
	\begin{align*}
	    &\omega_{(\mu,\nu)}(\widehat{s}) = 2(a^*(\mu)+a^*(\nu)-a(\mu)-a(\nu)), \\
	    &\omega_{(\mu,\nu)}(\widehat{t}) = |\mu|-|\nu|.
	\end{align*}
\end{prop}

Combining everything, this implies that the exponent of $q$ in the eigenvalue $\lambda_\chi$, where $\chi = \chi_{(\mu,\nu)} \in \mathrm{Irr}(W)$, equals
\begin{align}\label{eq:exponenttypeB}
\frac{n(n-1)}{2}+(a^*(\mu)+a^*(\nu)-a(\mu)-a(\nu))+e|\mu|,
\end{align}
and its sign is $(-1)^{|\nu|}$.

\begin{lemma}\label{lem:eigenvaluestypeB}
	The largest eigenvalue is obtained for $\chi_{(\mu,\nu)} = \chi_{([n],\emptyset)}$ and equals $q^{n(n+e-1)}$. The smallest eigenvalue is obtained for $\chi_{(\mu,\nu)} = \chi_{([n-1],[1])}$, except when $e \in \{0,1/2\}$ and $n$ is odd, then it is obtained from the character $\chi_{(\emptyset,[n])}$. The corresponding eigenvalues are $-q^{(n-1)(n+e-1)}$ and $-q^{n(n-1)}$ respectively. 
\end{lemma}

\begin{proof}
	Fix $k$ between 0 and $n$, and consider the pairs of partitions $(\mu,\nu)$ with $|\mu| = n-k$ and $|\nu|=k$. From \Cref{lemma:partitionmaxvalue} we can see that $([n-k],[k])$ has the highest value for the exponent \eqref{eq:exponenttypeB} among all such partitions. Varying over $k$, it follows from a short computation that the largest value is attained by $([n],\emptyset)$. The corresponding sign is $(-1)^0$, so that this character leads to the largest eigenvalue.
	
	The candidates for the second largest exponent are $([n-1,1],\emptyset)$ and $([n-k],[k])$ for $1 \leq k \leq n$. Remark that the first one always leads to a positive eigenvalue, while the second one only leads to a negative eigenvalue if $k$ is odd. Another short computation shows that $([n-1],[1])$ gives the largest value in \eqref{eq:exponenttypeB}, except when $e \in \{0,1/2\}$, as it is then found on the pair $(\emptyset,[n])$. The sign for the former is always negative, whereas the latter has sign $(-1)^n$ and hence only gives rise to a negative eigenvalue when $n$ is odd. When $e \in \{0,1/2\}$ and $n$ is even, the smallest eigenvalue is again found by the character $\chi_{([n-1],[1])}$.
\end{proof}

\begin{remark}
	When $e = 0$ and $n$ is even, we can deduce from this proof that $(\mu,\nu) = (\emptyset,[n])$ also leads to the maximal eigenvalue $q^{n(n-1)}$. This is to be expected: recall that when $n$ is even, generators can only be opposite to generators in the same class and hence the opposition graph is the disjoint union of two isomorphic graphs.
\end{remark}

This completes the proof of \Cref{thm:boundtypeBmaximal}. Again, the exponent of the largest eigenvalue equals $\ell(w_0)$, but quite literally with a twist. For the untwisted groups we have $e=1$ and the exponent matches $\ell(w_0)$ in \Cref{table:weylgroups} exactly. On the other hand, for the twisted groups, the length function is actually modified slightly. In this case, the generators $s_i$ have length 1, but $t$ has length $e$. Now all that one needs to know is that $w_0$ contains $n$ times the generator $t$, since $w_0 = t(s_1ts_1)(s_2s_1ts_1s_2)\dots(s_n\dots s_1ts_1 \dots s_n)$ \cite[Example 1.4.6]{GP00}

We can extend this upper bound to partial flags as well, by showing that the relevant characters appear in the decomposition of $\mathrm{ind}_{W_J}^W(1_{W_J})$, similarly as in the proof of \Cref{thm:boundtypeApartial} 

\begin{theorem}\label{thm:boundtypeBpartial}
	If $C$ is an EKR-set of flags of cotype $J$ in a polar space of rank $n$ and parameter $e$, then
	\[|C| \leq \frac{v}{1+q^{n+e-1}},\]
	where $v = [G:P_J]$ is the total number of flags of cotype $J$, except for $e \leq \frac{1}{2}, n \notin J$ and $n$ odd, where we can improve this bound to
	\[|C| \leq \frac{v}{1+q^{n/2}}, \text{ for } e = \frac{1}{2}, \]
	or
	\[|C| \leq \frac{v}{2}, \text{ for } e = 0.\]
\end{theorem}

\begin{proof}
	We will show that the characters $\chi_{([n],\emptyset)}$ and $\chi_{([n-1],[1])}$ survive upon restriction to any parabolic subgroup $W_J$ whereas $\chi_{(\emptyset,[n])}$ survives whenever $n$ is contained in the type of the flag, or equivalently, $n \notin J$.
	
	The irreducible character $\chi_{([n],\emptyset)}$ is the trivial character \cite[Section 5.5]{GP00}, and its restriction to any parabolic subgroup $W_J$ is hence the trivial character $1_{W_J}$, so that it appears with multiplicity 1 in $\mathrm{ind}_{W_J}^W(1_{W_J})$.
	
	Now define the the irreducible linear character $\chi$ such that $\chi(s_i) = 1$ for $1 \leq i \leq n-1$ and $\chi(t) = -1$. Then for any irreducible character $\chi_{(\mu,\nu)}$ we have $\chi \cdot \chi_{(\mu,\nu)} = \chi_{(\nu,\mu)}$ \cite[Theorem 5.5.6]{GP00}. In particular, $\chi_{(\emptyset,[n])}$ equals $\chi$. Therefore, its restriction to $W_J$ is trivial if and only if $n \notin J$. We conclude that it appears with multiplicity 1 in the decomposition if and only if $n \notin J$.
	
	Lastly, the character $\chi_{([n-1],[1])}$ is the irreducible character of the standard reflection representation of $W$ defined as follows: recall the standard representation of $\langle s_1,\dots, s_{n-1} \rangle = \Sym(n) \leq W$ on $\C^n = \langle e_1, \dots, e_n \rangle$. We can extend this to a representation of $W$ by defining $t \cdot e_n = -e_n$ and it turns out that the corresponding character is $\chi_{([n-1],[1])}$, see \cite[Section 1.4.1 and Proposition 5.5.7]{GP00}. Now the proof is again completely analogous as before: any parabolic subgroup $W_J$ is isomorphic to $\Sym(a_1+1) \times \dots \times \Sym(a_k+1) \times W_{b}$ for some positive integers $a_1, \dots, a_{k},b$ with $a_1 + \dots + a_{k} + b < n$ and where $W_b$ is a Weyl group of type $B$ and rank $b$. Then the module $\C^n$ of the standard representation decomposes into $W_J$-invariant submodules as $\C^{a_1} \oplus \dots \oplus \C^{a_k} \oplus \C^b \oplus U$, where the action of on $U$ is the trivial one. Since $\dim(U) = n-(a_1+\dots+a_k+b) > 0$, the result again follows. We remark that this also follows from \cite[Remark 6.3.7]{GP00}.
\end{proof}

\subsubsection{Reaching the upper bound}

We are now in a position to figure out the geometrical interpretation of opposition and show that it is related to previous literature. In particular, we will see that \Cref{thm:boundtypeBpartial} recovers and generalizes known bounds. A similar proof as for \Cref{lem:oppositiongeometricallytypeA} in type $A_n$ works here as well, but we can give a slightly different proof that is more inspired by the theory of buildings. 

\begin{lemma}
	Two maximal flags $F_1$ and $F_2$ are opposite if and only if $F_1 =\{\langle e_1 \rangle, \langle e_1,e_2 \rangle, \dots, \langle e_1,\dots,e_n \rangle \}$ and $F_2 = \{\langle e_{-1} \rangle, \langle e_{-1},e_{-2} \rangle, \dots, \langle e_{-1},\dots,e_{-n} \rangle \}$ for some polar frame $\{e_{-n},\dots,e_{-1},e_1,\dots,e_n\}$.
\end{lemma}

\begin{proof}
	By \cite[Theorem 10.4.6]{BC13} we know that for any two maximal flags $F_1$ and $F_2$ there exists a polar frame such that $F_1 =\{\langle e_1 \rangle, \langle e_1,e_2 \rangle, \dots, \langle e_1,\dots,e_n \rangle \}$ and $F_2 =\{\langle e_{i_1} \rangle, \langle e_{i_1},e_{i_2} \rangle, \dots, \langle e_{i_1},\dots,e_{i_n} \rangle \}$, where $\{i_1,\dots,i_n\} \subseteq \{-n,\dots,n\}$.
	
	In general, for any given polar frame there are exactly $2^nn!$ distinct maximal flags that can be constructed from the points in the polar frame. Fixing one such maximal flag $F$, one sees that all such maximal flags are exactly the flags $w \cdot F$ for $w \in W$ (where the action of $W$ on the polar frame induces an action on any maximal flag constructed from that frame). This means that if we identify $F$ with its stabilizer $B$, then every double coset $BwB$, $w \in W$, has a representative $wB$ among these maximal flags. 
	
	In our case, we can set $F = F_1$ and by assumption $F_2 = w_0 \cdot F$. Recalling that $w_0 = (1,-1)\dots(n,-n)$, it follows immediately that $F_2$ should be the flag $\{\langle e_{-1} \rangle, \langle e_{-1},e_{-2} \rangle, \dots, \langle e_{-1},\dots,e_{-n} \rangle \}$.
\end{proof}

Since the action of $w_0$ on the Dynkin diagram is trivial, a subspace of type $k$ can only be opposite to another subspace of type $k$. From the previous lemma and its proof, we immediately obtain the following corollary.

\begin{cor}\label{cor:oppositiongeometricallytypeB}
	Two $k$-dimensional subspaces $U$ and $V$ are opposite (in the group-theoretical sense) in a polar space if and only if no point of $U$ is collinear with all points of $V$ and vice versa.
\end{cor}

In particular, generators are opposite if they are disjoint. This implies for example that the bounds in \Cref{thm:boundtypeBpartial} for $J = S \setminus \{n\}$ are bounds for EKR-sets of generators as studied in \cite{Sta80} and \cite[Theorem 9]{PSV11}. For other types of flags in polar spaces, there are very few results with the notable exceptions of \cite{IMM18} which studies flags of type $\{n-1\}$ in a polar space of rank $n$, $n$ even, and parameter $e=0$ and \cite{Metsch192} which investigates flags of type $\{2\}$ in all polar spaces. We reiterate that the results in \cite{Metsch16} for $k$-spaces in polar space of rank $n$, $1 \leq k \leq n$, are independent of ours, as the `far away' relation in that paper is defined as having empty intersection. This is not the same as being opposite.

There are two obvious examples that reach the upper bound.

\begin{theorem}\label{thm:sharpexamplestypeB}
	Let $C$ be the set of flags of type $J$ in a polar space of rank $n$ and parameter $e$ such that either
	\begin{itemize}
		\item [(i)] $1 \in J$ and every flag in $C$ has its point in a fixed generator,
		\item [(ii)] $n \in J$ and every flag in $C$ has its generators through a fixed point for $e \geq 1$ or $n$ even,
		\item [(iii)] $n \in J$ and every flag in $C$ has its generator in a fixed class for $e=0$ and $n$ odd,
	\end{itemize}
	then $C$ is an EKR-set of flags that meets the upper bound \Cref{thm:boundtypeBpartial}.
\end{theorem}
\begin{proof}
	Combining \Cref{prop:numberofspacestypeB} and \Cref{cor:oppositiongeometricallytypeB} shows the result in the first case. In the second case, it is known that the number of generators in a polar space of rank $n$ through a fixed point equals the number of generators in a polar space of rank $n-1$, with the same parameter $e$.
	
	For the last case: the generators on a hyperbolic quadric fall into two classes of equal size \cite[Lemma 7.8.4]{BC13}, depending on their intersections: two generators $N_1,N_2$ are in the same class if and only if $\dim(N_1 \cap N_2) \equiv n \pmod{2}$. This means that when $n$ is odd, one class of generators is an EKR-set of type $\{n\}$, as they can never be disjoint.	
\end{proof}

These are not the only examples meeting the upper bound. In \cite{PSV11} it is shown that there could be very different examples, depending on the polar space. There is no construction known attaining the upper bound for type $\{n\}$ when $e = 1/2$ and $n$ odd. For $n = 3$, it is shown in \cite[Theorem 45]{PSV11} that the upper bound cannot be attained and we strongly believe this holds for all $n \geq 3$.

\begin{prob}
	Do there exist other EKR-sets of type $J$, with $1 \in J$ or $n \in J$ that meet the upper bound in \Cref{thm:boundtypeBpartial} but are not contained in the construction in \Cref{thm:sharpexamplestypeB} or by blowing up examples in \cite{PSV11}?
\end{prob}

\begin{prob}
	Do there exist EKR-sets of $k$-spaces in a polar space of rank $n$ with $2 \leq k \leq n-1$ meeting the upper bound in \Cref{thm:boundtypeBpartial}?
\end{prob}

\subsubsection{Eigenvalues of single element flags}

Before we move on to the groups of type $D_n$, we would like to indicate the strength of \Cref{algorithmpartial} by calculating all eigenvalues of oppositeness when the type is a single element, as was also done in \cite{Eisfeld99} and \cite{Vanhove11}. We emphasize that while their calculations are rather technical, ours rely in essence on elementary combinatorics of Young diagrams and \Cref{algorithmpartial}.

\begin{defin}
	A \textbf{Young diagram} of a partition $\mu = [\mu_1,\dots,\mu_k]$  is a collection of $|\mu|$ boxes, arranged in left-justified rows, with $\mu_i$ boxes in the $i$-th row.
\end{defin}

In type $B_n$, every irreducible character hence corresponds to a pair of Young diagrams with a total of $n$ boxes. We will compute all eigenvalues for $k$-dimensional subspaces in a polar space of rank $n$ and parameter $e$. The main tool we need for the first step in \Cref{algorithmpartial} is Pieri's rule for type $B_n$ \cite[6.1.9]{GP00}. Let $J = S \setminus \{s_k\}$, then

\[\mathrm{ind}^W_{W_J}(1_{W_J}) = \sum_{(\mu,\nu)}\chi_{(\mu,\nu)},\]
where $(\mu,\nu)$ runs over all pairs of partitions such that, for some $d$ with $0 \leq d \leq k$,

\begin{itemize}
	\item the Young diagram of $\mu$ can be obtained from the Young diagram of the partition $[n-k]$ by adding $d$ boxes, no two in the same column;
	\item the Young diagram of $\nu$ can be obtained from the empty Young diagram by adding $k-d$ boxes, no two in the same column.
\end{itemize}

Therefore, the only options for $(\mu,\nu)$ are for any $0 \leq d \leq k$:
\begin{align*}
	&\mu = [n-k+d-i,i], \hspace{0.5cm} 0 \leq i \leq \min(n-k,d), \\
	&\nu = [k-d].
\end{align*}

This means that 

\[\mathrm{ind}^W_{W_J}(1_{W_J}) = \sum_{d=0}^k\sum_{i=0}^{\min(n-k,d)}\chi_{([n-k+d-i,i],[k-d])},\]
which completes the first step in \Cref{algorithmpartial}. Now consider the character $\chi = \chi_{([n-k+d-i,i],[k-d])}$. In the previous sections, we have seen that the corresponding exponent can be easily computed in terms of the two partitions and equals 

\[\frac{n(n-1)}{2}+(a^*(\mu)+a^*(\nu)-a(\mu)-a(\nu))+e|\mu| = \frac{n(n-1)}{2}+{n-k+d-i \choose 2} + {i \choose 2} + {k-d \choose 2} - i + (n-k+d)e.\]

The sign on the other hand, equals $(-1)^{k-d}$. Finally, the length of the longest word in $W_J$ can be computed from \Cref{table:weylgroups} as $W_J \cong \Sym(k) \times B_{n-k}$. Recalling that the parameter $e$ plays a role in the length function in type $B_n$, we obtain $\ell =  {k \choose 2}+(n-k)(n-k-1+e)$.

\begin{prop}
	The eigenvalues of oppositeness of $k$-dimensional spaces in a polar space of rank $n$ and parameter $e$ are 
	\[(-1)^{k-d}q^{n(d+k-i)+(k-d)(k-d+i)-k(3k+1)/2+i(i-1)+de},\]
	where $0 \leq d \leq k$, $0 \leq i \leq \min(n-k,d)$.
\end{prop} 

To recover the exact values in \cite[Theorem 4.3.15]{Vanhove11} (with our naming $k$ for the dimension of the space and $n$ for the rank instead of $n$ and $d$ respectively), the variable $i$ should stay the same but we make the change $d \rightarrow k+i-r$. It is a nice exercise to see that not only the exponents then become equal, but the range of the variables  also changes accordingly.

\subsection{Classical groups with Weyl group $D_n$, $n \geq 4$}

We have already seen that the generators of a hyperbolic quadric fall into two classes, depending on their mutual intersection. Considering this, we can modify the polar space of rank $n$ with parameter $e=0$, which can be considered as building of type $B_n$, to a building of type $D_n$ by making the distinction and constructing the oriflamme complex as described in \cite[Definition 7.8.5]{BC13}. Essentially, most of the incidence relation remains unchanged, but two new types $n$ and $n'$ are defined, corresponding to the two classes of generator, and the type $n-1$ is omitted. Two subspaces of type $n$ and $n'$ are incident if they intersect in a subspace of dimension $n-1$. 

For EKR-sets of flags, there is hence very little difference in considering the hyperbolic quadric as a polar space with parameter $e=0$ or by investigating its oriflamme complex. Since generators of the same type can only be opposite if $n$ is even (recall the action of $w_0$ on the diagram), this is in fact the only time when a distinction could be made. When $n$ is even, we can consider flags of type $\{n\}$ or $\{n'\}$ separately. The union of the two corresponding oppositeness graphs is then the oppositeness graph we would obtain by considering flags of type $\{n\}$ in type $B_n$.

Nevertheless, for the sake of completeness, we will run \Cref{algorithmmaximal} and \Cref{algorithmpartial} in this case as well and see that we obtain in almost all cases the same bounds as in \Cref{thm:boundtypeBmaximal} and \Cref{thm:boundtypeBpartial}

\subsubsection{The upper bound}

For this section, denote the Weyl group of type $D_n$ as $WD_n$ and that of type $B_n$ as $WB_n$. Then $WD_n$ is the index two subgroup generated by $\{s_1,\dots,s_{n-1},u:=ts_{n-1}t\}$ in $WB_n$. One can indeed check that $us_{n-1} = s_{n-1}u$, $us_{n-2}u = s_2us_{n-2}$ and $us_i = s_iu$ for $1 \leq i \leq n-3$. Equivalently, from the permutation group definition of $WB_n$, it is the subgroup with an even number of sign flips. Its corresponding Iwahori-Hecke algebra is a subalgebra of the Iwahori-Hecke algebra of type $B_n$ with structure constants $q_s = q$ and $q_t = 1$ \cite[Section 10.4]{GP00}. This again makes sense in the geometrical side of the picture, considering the hyperbolic quadric as a polar space with parameter $e=0$.

It follows that the irreducible characters of $WD_n$, and hence the eigenvalues of oppositeness, follow almost directly from the discussion in the previous section. The irreducible characters of $WB_n$ are indexed by pairs of partitions $(\mu,\nu)$ and two possibilities arise \cite[Section 5.6]{GP00}:

\begin{itemize}
	\item $\mu \neq \nu$ and the restriction of $\chi_{(\mu,\nu)}$ to $WD_n$ remains irreducible and equals the restriction of $\chi_{(\nu,\mu)}$ or,
	\item $\mu = \nu$ and the restriction of $\chi_{(\mu,\mu)}$ is the sum of two irreducible characters, denoted as $\chi_{(\mu,+)}$ and $\chi_{(\mu,-)}$.
\end{itemize}

Therefore, the restriction of $\chi_{([n],\emptyset)}$ and $\chi_{(\emptyset,[n])}$ coincide so that we immediately find the following results.

\begin{lemma}
	The largest eigenvalue is obtained for $\chi_{(\mu,\nu)} = \chi_{([n],\emptyset)}$ and equals $q^{n(n-1)}$. The smallest eigenvalue is obtained for $\chi_{(\mu,\nu)} = \chi_{([n-1],[1])}$ and equals $-q^{(n-1)^2}$. 
\end{lemma}

\begin{proof}
	When $n$ is even, the longest word in $WB_n$ contains an even number of $t$ in any word in the generators, so it belongs to $WD_n$. In this case, everything follows from \Cref{lem:eigenvaluestypeB}.
	
	However, when $n$ is odd, the longest word in $WD_n$ is $tw^B_0$, where $w^B_0$ is the longest word in $WB_n$. In the standard reflection representation, the image of $t$ is the diagonal matrix $\mathrm{diag}(1,\dots,1,-1)$, and the image of $w^B_0$ is $\mathrm{diag}(-1,\dots,-1)$. It follows that $\chi_{([n-1],[1])}(tw^B_0) = -n+2$ and hence the corresponding eigenvalue appears with both signs.
\end{proof}

The number of subspaces of type $k$ in a hyperbolic quadric is given by the same formula as in \Cref{prop:numberofspacestypeB}, with $e=0$, except for $k = n,n'$, when it is half of the given number. In summary, this gives the following result.

\begin{theorem}\label{thm:boundtypeDmaximal}
	If $C$ is an EKR-set of maximal flags in a hyperbolic quadric then
	\[|C| \leq \frac{\prod_{i=1}^{n-1}(q^{i}+1)\qbinom{n}{1}\dots \qbinom{1}{1}}{1+q^{n-1}} = \prod_{i=1}^{n-2}(q^{i}+1)\qbinom{n}{1}\dots \qbinom{1}{1}.\]
\end{theorem}

Recall that for partial flags, we need to take care of the parity of $n$, as this has an effect on $w_0$ and its action on $S$. When $n$ is even, the action is trivial as $w_0 = w_0^B$ is central, and oppositeness is defined for any flag of type $J$, $J \subseteq S$. On the other hand, when $n$ is odd, the end nodes $s_n$ and $u$ are switched under the action of $w_0 = tw_0^B$ so that self-opposite flags either contain no $n$-dimensional subspaces or contain both classes of $n$-dimensional subspaces. 

\begin{theorem}\label{thm:boundtypeDpartial}
	If $C$ is an EKR-set of flags of cotype $J$ in the oriflamme complex of a hyperbolic quadric, such that $J^{w_0} = J$, then
	\[|C| \leq \frac{v}{1+q^{n-1}},\]
	where $v = [G:P_J]$ is the total number of flags of cotype $J$.
\end{theorem}

\begin{proof}
	The trivial character and the character of the standard reflection representation both appear in $\mathrm{ind}(1_{W_J}^W)$ for any choice of nonempty $J \subseteq S$. The former as the trivial character restricts to the trivial character, the latter by the same reasoning as before, or by \cite[Remark 6.3.7]{GP00}. 
\end{proof}

%

\subsubsection{Reaching the upper bound}

A maximal flag in type $D_n$ contains two $n$-dimensional spaces. For every maximal flag, we can find a polar frame $\{e_{-n},\dots,e_{n}\}$ in the hyperbolic quadric as before, such that the maximal flag is given by $\{\langle e_1 \rangle, \dots, \langle e_1,\dots,e_{n-2} \rangle, \langle e_1,\dots,e_{n-1},e_n \rangle, \langle e_1,\dots,e_{n-1},e_{-n} \rangle \}$. The two generators $\langle e_1,\dots,e_{n-1}, e_n \rangle$ and $\langle e_1,\dots,e_{n-1},e_{-n} \rangle$ are the only generators incident with the $(n-1)$-dimensional space $\langle e_1,\dots,e_{n-1} \rangle$ in the hyperbolic quadric. Using this, we can see that the group-theoretical notion of opposition corresponds to the usual geometrical notion in polar spaces, as in type $B_n$.

\begin{lemma}\label{lem:oppositiongeometricallytypeD}
	Two maximal flags $F_1$ and $F_2$ are opposite if and only if 
	\begin{align*}
	F_1 &= \{\langle e_1 \rangle, \dots, \langle e_1,\dots,e_{n-2} \rangle, \langle e_1,\dots,e_{n-1},e_n \rangle, \langle e_1,\dots,e_{n-1},e_{-n} \rangle \}  \\
	F_2 &= \{\langle e_{-1} \rangle, \dots, \langle e_{-1},\dots,e_{-(n-2)} \rangle, \langle e_{-1},\dots,e_{-(n-1)},e_{(-1)^{n-1}n} \rangle, \langle e_{-1},\dots,e_{-(n-1)},e_{(-1)^{n}n} \rangle\}
	\end{align*} 
	for some polar frame $\{e_{-n},\dots,e_{-1},e_1,\dots,e_n\}$.
\end{lemma}

\begin{proof}
	This follows in a similar way as \Cref{cor:oppositiongeometricallytypeB}, with the exception that the action of $w_0$ depends on the parity of $n$. When $n$ is even, we know that $w_0 = (1,-1)\dots(n,-n)$ so that $w_0 \cdot e_n = e_{-n}$ and vice versa. When $n$ is odd, we have seen that $w_0 = (1,-1)\dots(n-1,-(n-1))$ and the result follows.
\end{proof}

\begin{cor}\label{cor:oppositiongeneratorstypeD}
	Two generators are opposite if and only if they are disjoint and this can only occur in the following cases: 
	\begin{itemize}
		\item $n$ is even and they are of the same type,
		\item $n$ is odd and they are of different type.
	\end{itemize}
\end{cor}

Using these properties, we can find some simple examples of EKR-sets meeting the bound in \Cref{thm:boundtypeDpartial}.

\begin{theorem}\label{lem:sharpexamplestypeD}
	Let $C$ be the set of flags of type $J = J^{w_0}$ in the oriflamme complex of a hyperbolic quadric such that
	\begin{itemize}
		\item [(i)] $1 \in J$ and every flag in $C$ has its point in a fixed generator,		
		\item [(ii)] $n \in J$ (resp. $n' \in J$), $n$ is even, and every flag in $C$ has its subspace of type $n$ (resp. $n'$) through a fixed point,
		\item [(iii)] $n \in J$ (resp. $n' \in J$), $n=4$, and every flag in $C$ has its subspace of type $n$ incident with a fixed subspace of type $n'$ (resp. $n$). 
	\end{itemize}
	then $C$ is an EKR-set of flags that meets the upper bound \Cref{thm:boundtypeDpartial}.
\end{theorem}

\begin{proof}
	The first case is seen quite directly. For the second case, this follows from \Cref{cor:oppositiongeneratorstypeD} and \Cref{prop:numberofspacestypeB}. The third case follows from the second by the triality automorphism of $D_4$, which corresponds to the unique diagram automorphism of order 3.
\end{proof}

Flags of type $\{n,n'\}$, $n$ even, were investigated in \cite{IMM18} where it was shown that the examples described above are the only ones attaining equality.

\begin{prob}
	Do there exist other EKR-sets of type $J$, with $1 \in J$ or $n \in J$ that meet the upper bound in \Cref{thm:boundtypeDpartial} but are not contained in the construction in \Cref{lem:sharpexamplestypeD}?
\end{prob}

\subsection{Conclusion and further research}

Now that we have managed to obtain upper bounds for EKR-sets of flags, the next step would be to investigate the case of equality. Often, the multiplicity of the smallest eigenvalue is a useful piece of information, which we have not mentioned so far. To compute this, a little more effort and understanding of Iwahori-Hecke algebras is required, but the key term is that of `generic degrees', see for example \cite{Hoefsmit74}. These are known for the classical types and combinatorial formulae exist for them. In case of the standard reflection representation, they are listed in \cite[Example 10.5.8]{GP00}. \\

Throughout the previous sections, one observes that the largest eigenvalue always comes from the trivial character. This is no surprise, as the valency of $A_{w_0}$ is the largest eigenvalue, which has the all-one vector as eigenvector. What is remarkable though, is that the standard reflection representation leads to the smallest eigenvalue in almost all cases. As nice as it is, we do not have a good explanation for this phenomenon, and we would like to see one. \\

Another interesting direction of research is to relax the notion of opposition by omitting the condition on the type of the flags. Perhaps the framework of Iwahori-Hecke algebras allows one to obtain results as well, in particular the computation of the eigenvalues of the corresponding Kneser graphs. \\

Finally, we can say something about the exceptional types. The algorithms we described work equally well in these cases. The work is however severely easier as the Weyl groups are groups of fixed size. In theory, a computer could do all the work, but if one wants to do computations by hand, the tables in \cite[Appendix C]{GP00} will prove to be very useful. However, we are not aware of any situation where the Delsarte-Hoffman bound is sharp. For example, in \cite[p129]{Guven12} several bounds are given for single element flags in exceptional types. One of them is a bound of order $q^9$ for points, i.e. type $\{1\}$, in the building with Weyl group $F_4$. Compare this to the results in $\cite{Metsch19}$, where a sharp upper bound of order $q^7$ is found. It remains to be seen if these bounds can prove to be useful for the exceptional geometries.

\noindent
\textbf{Address of the authors}\\
\texttt{\{Jan.De.Beule, Sam.Mattheus\}@vub.ac.be}\\
Department of Mathematics,
Vrije Universiteit Brussel,\\
Pleinlaan 2,
B-1050 Elsene,
Belgium\\

\noindent
\texttt{Klaus.Metsch@math.uni-giessen.de}\\
Justus-Liebig-Universität, 
Mathematisches Institut,\\ 
Arndtstra\ss e 2, 
D-35392, Gie\ss en, 
Germany

\end{document}